\documentclass[10]{amsart}
\usepackage{url}
\usepackage{latexsym}
\usepackage{float}
\usepackage{amssymb}
\usepackage{tikz}
\usepackage{xypic}
\usepackage{MnSymbol}
\usepackage{caption}
\usepackage[vcentermath]{youngtab}

\newtheorem{thm}{Theorem}[section]

\newtheorem{lem}[thm]{Lemma}
\newtheorem{cor}[thm]{Corollary}

\newtheorem{pro}[thm]{Proposition}

\theoremstyle{remark}

\newtheorem{ex}[thm]{Example}
\newtheorem{rem}[thm]{Remark}

\theoremstyle{definition}
\newtheorem{Def}[thm]{Definition}
\newtheorem{notation}[thm]{Notation}

\numberwithin{equation}{section}

\newcommand{\R}{\mathbb{R}}           
\newcommand{\C}{\mathbb{C}}           
\newcommand{\Z}{\mathbb{Z}}           

\newcommand{\rank}{\text{rank}}

\newcommand{\ad}{\text{ad}}
\newcommand{\KS}{\mathrm{KS}}

\newcommand{\AC}{\mathcal{AC}}
\newcommand{\ACC}{\mathrm{AC}}

\newcommand{\fb}{{\mathfrak b}}

\newcommand{\fg}{{\mathfrak g}}

\newcommand{\fk}{{\mathfrak k}}
\newcommand{\fl}{{\mathfrak l}}

\newcommand{\fn}{{\mathfrak n}}

\newcommand{\fp}{{\mathfrak p}}

\newcommand{\fq}{{\mathfrak q}}
\newcommand{\ft}{{\mathfrak t}}
\newcommand{\fu}{{\mathfrak u}}

\newcommand{\fs}{{\mathfrak s}}

\newcommand{\be}{\begin{equation}}
\newcommand{\beu}{\begin{equation*}}

\newcommand{\GLnC}{\mathrm{GL}_n(\C)}

\newcommand{\glnC}{\mathfrak{gl}_n(\C)}

\newcommand{\SLnR}{\mathrm{SL}_n(\R)}
\newcommand{\slnC}{\mathfrak{sl}_n(\C)}
\newcommand{\slnR}{\mathfrak{sl}_n(\R)}

\newcommand{\slR}[1]{\mathfrak{sl}_{#1}(\R)}
\newcommand{\soC}[1]{\mathfrak{so}_{#1}(\C)}

\newcommand{\spC}[1]{\mathfrak{sp}_{#1}(\C)}

\makeatletter
\newcommand\xleftrightarrow[2][]{%
  \ext@arrow 9999{\longleftrightarrowfill@}{#1}{#2}}
\newcommand\longleftrightarrowfill@{%
  \arrowfill@\leftarrow\relbar\rightarrow}
\makeatother

\begin{document}

\baselineskip=16pt
\sloppy
\title[Approximation of nilpotent orbits for simple Lie groups]
{Approximation of nilpotent orbits for simple Lie groups}

\author{Lucas Fresse}
\author{Salah Mehdi}
\address{Institut Elie Cartan de Lorraine, CNRS - UMR 7502, Universit\'e de Lorraine - France}
\email{lucas.fresse@univ-lorraine.fr}
\email{salah.mehdi@univ-lorraine.fr}

\begin{abstract}
We propose a systematic and topological study of limits $\lim_{\nu\to 0^+}G_\mathbb{R}\cdot(\nu x)$ of continuous families of adjoint orbits for a non-compact simple real Lie group $G_\R$.
This limit is always a finite union of nilpotent orbits. We describe explicitly these nilpotent orbits in terms of Richardson orbits in the case of hyperbolic semisimple elements.
We also show that one can approximate minimal nilpotent orbits or even nilpotent orbits by elliptic semisimple orbits. The special cases of $\mathrm{SL}_n(\mathbb{R})$ and $\mathrm{SU}(p,q)$ are computed in detail.
\end{abstract}

\maketitle

\section{Introduction}

\subsection{Continuous families of adjoint orbits and their limits}

The structure of a connected real Lie group $G_{\mathbb R}$ is closely related with the topology and the geometry of its adjoint orbits on its Lie algebra $\fg_\R$. For instance, when $G_{\mathbb R}$ is Abelian, its adjoint orbits are singletons, whereas, when $G_{\mathbb R}$ is simple, adjoint orbits are symplectic manifolds. If $G_{\mathbb R}$ is compact, its adjoint orbits are compact symplectic manifolds. For example, each adjoint orbit of the $3$-dimensional special unitary group $G_{\mathbb R}=\mathrm{SU}(2)$ takes the form $G_{\mathbb R}\cdot (r x)$, where $x=\begin{pmatrix} i&0\\ 0&-i\end{pmatrix}\in\fg_{\mathbb R}={\mathfrak s}{\mathfrak u}(2)$ and $r\geq 0$, and it is diffeomorphic with the sphere of radius  $r\geq 0$.
%
\begin{figure}[H]
\includegraphics[scale=0.55]{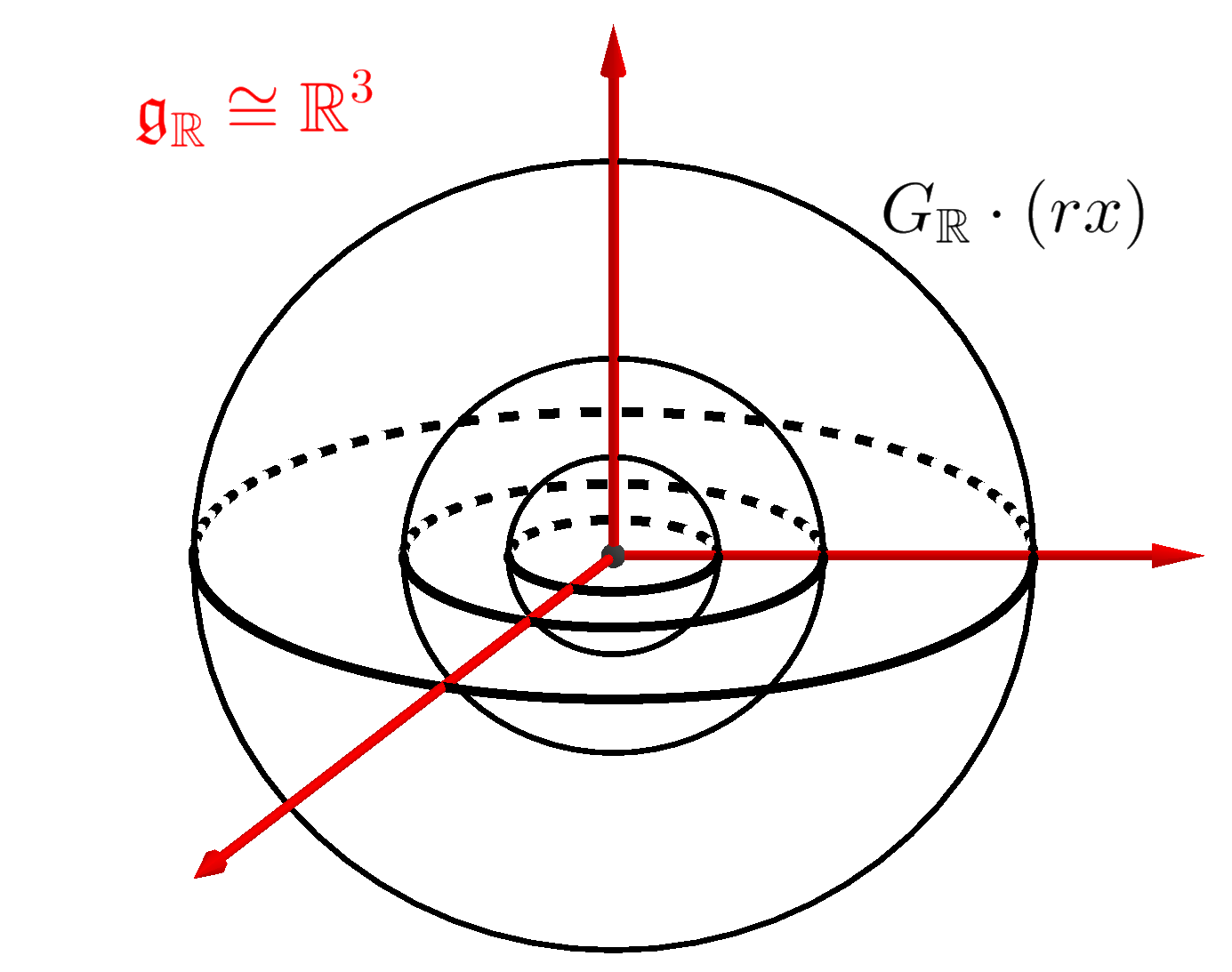}
\caption{Adjoint orbits for $\mathrm{SU}(2)$.}
\label{figure-1}
\end{figure}
%
The picture in Figure \ref{figure-1} suggests that the continuous family $\{G_{\mathbb R}\cdot (r x)\}_{r>0}$ converges to the trivial orbit in the following sense: 
\begin{equation}
\label{limit-1.1}
\lim_{r\to 0^+}G_\mathbb{R}\cdot(r x):=\bigcap_{\epsilon>0}\overline{\bigcup_{r\in(0,\epsilon)}G_\mathbb{R}\cdot(r x)}=\{0\}.
\end{equation}
It turns out that this true for any $x$, whenever $G_{\mathbb R}$ is compact and simple (see Remark \ref{remcpt}). 

The non-compact case is more subtle. For instance, the nontrivial adjoint orbits of the $3$-dimensional special linear group $G_{\mathbb R}=\mathrm{SL}_2({\mathbb R})$ split into three families (see Example \ref{E1}): 
\begin{itemize}
\item[$\bullet$] $G_{\mathbb R}\cdot(\lambda f^\pm)$, $\lambda>0$, where $f^\pm=\pm\begin{pmatrix}0&1\\ -1&0\end{pmatrix}$, which identifies with the upper sheet $z\geq 0$, in the case of $f^+$ (resp., the lower sheet $z\leq 0$, in the case of $f^-$) of the two-sheeted hyperboloid $z^2-x^2-y^2=\lambda^2$;
\item[ $\bullet$]$G_{\mathbb R}\cdot(s\varphi)$, $s>0$, where $\varphi=\begin{pmatrix}1&0\\ 0&-1\end{pmatrix}$, which identifies with the one-sheeted hyperboloid $z^2-x^2-y^2=-s^2$;
\item[ $\bullet$]$G_{\mathbb R}\cdot\psi^{\pm}$, where $\psi^\pm=\begin{pmatrix}0&\frac{1}{2}\pm\frac{1}{2}\\ \frac{1}{2}\mp\frac{1}{2}&0\end{pmatrix}$, which identifies with the connected component $z>0$ (resp., $z<0$) of the cone $z^2-x^2-y^2=0$.
\end{itemize}
%
\begin{figure}[H]
\includegraphics[scale=0.5]{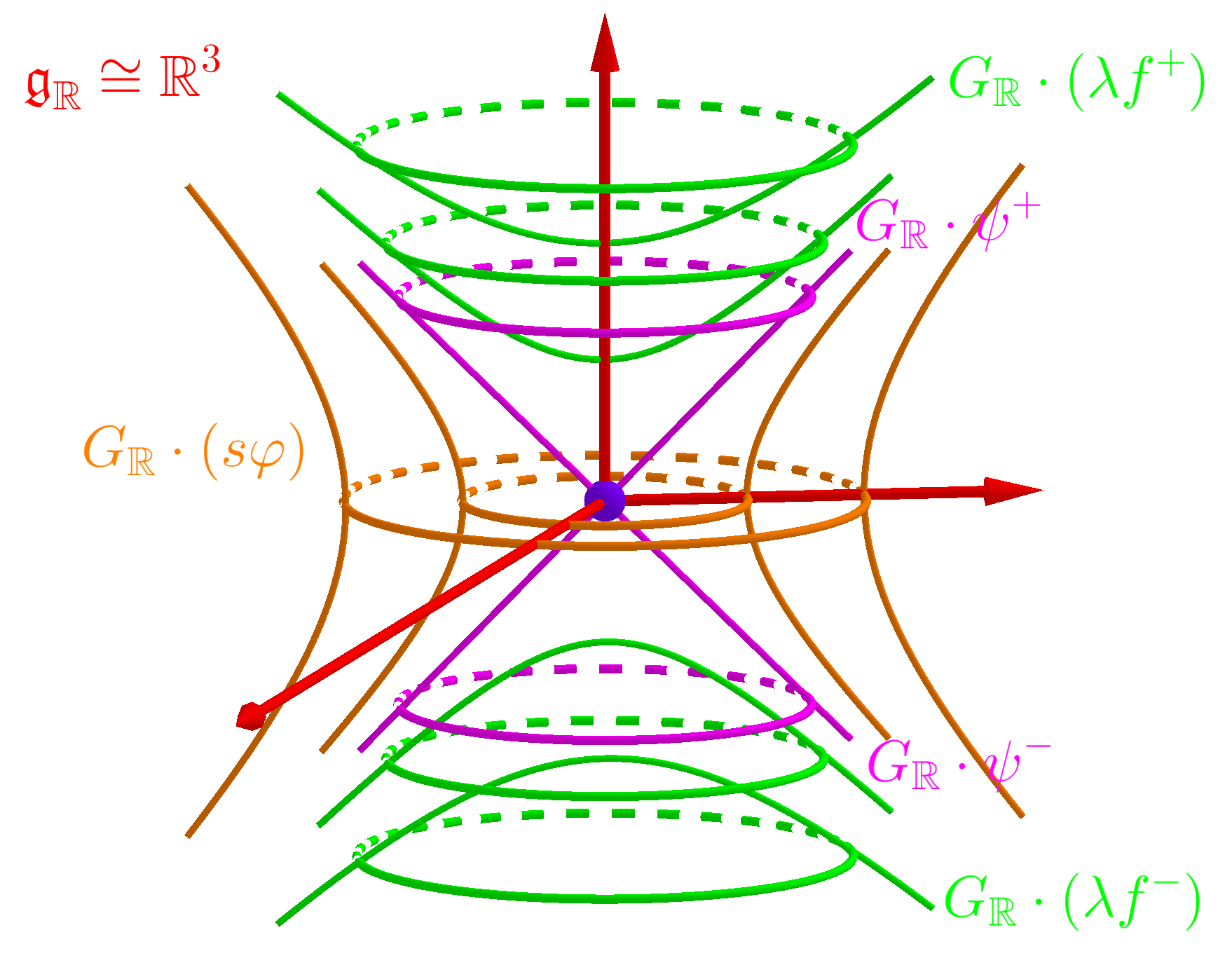}
\caption{Adjoint orbits for $\mathrm{SL}_2({\mathbb R})$.}
\label{I1}
\end{figure}
%

By the definition, these orbits are homogeneous spaces for $G_{\mathbb R}=\mathrm{SL}_2({\mathbb R})$. More precisely, if $P_{\mathbb R}=M_{\mathbb R}A_{\mathbb R}N_{\mathbb R}$ is the Langlands decomposition of a minimal parabolic subgroup of $G_{\mathbb R}$, $K_{\mathbb R}$ a maximal compact subgroup of $G_{\mathbb R}$ and $N_{\mathbb R}^-$ the opposite of $N_{\mathbb R}$ (i.e., defined by the opposite of the positive system of restricted roots), then one has the following diffeomorphisms:
\begin{equation*}
G_{\mathbb R}\cdot f^\pm\simeq G_{\mathbb R}/K_{\mathbb R},\;\;\;G_{\mathbb R}\cdot \varphi\simeq G_{\mathbb R}/M_{\mathbb R}A_{\mathbb R},\;\;\;
G_{\mathbb R}\cdot\psi^{+}\simeq G_{\mathbb R}/M_{\mathbb R}N_{\mathbb R}\;\;\text{ and }\;\;G_{\mathbb R}\cdot\psi^{-}\simeq G_{\mathbb R}/M_{\mathbb R}N_{\mathbb R}^-.
\end{equation*}
In particular, besides of being symplectic homogeneous manifolds, the orbits $G_{\mathbb R}\cdot f^\pm$ and $G_{\mathbb R}\cdot \varphi$ are respectively Hermitian Riemannian and pseudo-Riemannian symmetric spaces, while the orbits $G_{\mathbb R}\cdot\psi^{\pm}$ are neither symmetric nor reductive. Moreover, using the limit defined as in (\ref{limit-1.1}), one checks that (see Example \ref{E1}): 
\begin{eqnarray*}
\lim_{\lambda\to 0^+}G_\mathbb{R}\cdot(\lambda f^\pm)&=&\overline{G_{\mathbb R}\cdot\psi^\pm}\;\;=\;\;\lim_{\nu\to 0^+}G_\mathbb{R}\cdot(\nu\psi^\pm),\\
\lim_{s\to 0^+}G_\mathbb{R}\cdot(s\varphi)&=&\overline{G_{\mathbb R}\cdot\psi^+\cup G_{\mathbb R}\cdot\psi^-},
\end{eqnarray*}
where $\overline{G_{\mathbb R}\cdot\psi^\pm}$ (resp., $\overline{G_{\mathbb R}\cdot\psi^+\cup G_{\mathbb R}\cdot\psi^-}$) denotes the Zariski closure of $G_{\mathbb R}\cdot\psi^\pm$ (resp., $G_{\mathbb R}\cdot\psi^+\cup G_{\mathbb R}\cdot\psi^-$) in ${\mathfrak g}_{\mathbb R}$. 
In other words, the continuous family
$\{G_\mathbb{R}\cdot(\lambda f^+)\}_{\lambda >0}$ (resp., $\{G_\mathbb{R}\cdot(\lambda f^-)\}_{\lambda >0}$) of Riemannian manifolds converges towards the singular variety defined as the closure $\overline{G_{\mathbb R}\cdot\psi^+}$ (resp., $\overline{G_{\mathbb R}\cdot\psi^-}$) of the nilpotent orbit $G_{\mathbb R}\cdot\psi^+$ (resp., $G_{\mathbb R}\cdot\psi^-$).

Nilpotent orbits of non-compact simple Lie groups $G_{\mathbb R}$ play an important role in Mathematical Physics in the following sense. A general space time with symmetry group $G_{\mathbb R}$ is a homogeneous space $G_{\mathbb R}/H_{\mathbb R}$, where $H_{\mathbb R}$ is a closed subgroup of $G_{\mathbb R}$. The cotangent bundle $T^{\star}(G_{\mathbb R}/H_{\mathbb R})$ is naturally equipped with a structure of a symplectic manifold on which the group $G_{\mathbb R}$ acts in such a way that there is a $G_{\mathbb R}$-equivariant moment map
$$J:T^{\star}(G_{\mathbb R}/H_{\mathbb R})\rightarrow{\mathfrak g}_{\mathbb R}^*\cong\fg_\R.$$
In particular, the image of $J$ is a union of (co)adjoint orbits for $G_{\mathbb R}$. Let $\pi:T^{\star}(G_{\mathbb R}/H_{\mathbb R})\rightarrow G_{\mathbb R}/H_{\mathbb R}$ be the natural projection. Then, following Souriau \cite[\S14--15]{Sou70}, an orbit ${\mathcal O}=G_\R\cdot x$ in the image of $J$ is said to be the classical phase space of a free particle moving on $G_{\mathbb R}/H_{\mathbb R}$ if the projection $\pi(J^{-1}(x))$ of $J^{-1}(x)$ is a timelike or a lightlike geodesic on $G_{\mathbb R}/H_{\mathbb R}$ (assuming that geodesics exist on $G_\R/H_\R$). In this picture, a semisimple orbit describes the classical dynamics of a {\it massive} particle, while a nilpotent orbit describes the classical dynamics of a {\it massless} particle. However, unlike {\it massive} particles, it is not known how to quantize canonically classical dynamics of {\it massless} particles. In this respect, a systematic approximation of nilpotent orbits by semisimple orbits could help to understand better the quantization of {\it massless} particles.

On the other hand, limits of semisimple (most often elliptic) orbits are used in Representation Theory to bridge objects of different nature such as associated varieties, wave front sets, and characters, see \cite{BV,BV2,Bozicevic,Harris,HHO,KO,KO1,Rossmann80,Rossmann} and references therein. However, these limits are used as a tool and, to our best knowledge, there is no systematic study of the limit $\lim\limits_{\nu\to 0^+}G_\mathbb{R}\cdot(\nu x)$ for an arbitrary non-compact simple Lie group $G_{\mathbb R}$ and any $x\in\fg_{\mathbb R}$, which is based on topological arguments only.

\subsection{Outline of the paper}
This paper aims to provide a self-contained and systematic topological study of the limit of adjoint semisimple orbits
for connected non-compact simple linear real Lie groups. We also address the reverse problem of realizing (the closure of) a prescribed nilpotent orbit as the limit of a specific family of semisimple orbits. In this respect, our main results appear in Sections \ref{hyperbolic}--\ref{elliptic}.

Let us describe in more detail the content of the paper.
For the convenience of the reader, in Section \ref{preliminaries}, we collect basic definitions and properties about Jordan decompositions, elliptic and hyperbolic elements, ${\mathfrak s}{\mathfrak l}_2$-triples and the classification of complex nilpotent orbits. 

The definition and first properties of limits of orbits are given in Section \ref{limitorbit}. We  first observe that ${\lim\limits_{\nu\to 0^+}G_\mathbb{R}\cdot(\nu x)}$ is a finite union of nilpotent orbits which contains the limit $\lim\limits_{\nu\to 0^+}G_\mathbb{R}\cdot(\nu x_s)$ (resp., the Zariski closure $\overline{G_{\mathbb R}\cdot x_n}$) of the semisimple (resp., nilpotent) part of $x$ in its Jordan decomposition (Proposition \ref{basicprop}). In fact, every nilpotent orbit appears in the limit set 
$\lim\limits_{\nu\to 0^+}G_\mathbb{R}\cdot(\nu x)$ of a semisimple (elliptic or hyperbolic) element $x\in\fg_{\mathbb R}$ (Proposition \ref{P-existence}). 
This supports the idea that limits of orbits can be used to approximate nilpotent orbits by semisimple orbits.
Our main concern will be then to study to what extent a nilpotent orbit $\mathcal{O}$ can be approached by a continuous family of semisimple orbits $\{G_\R\cdot(\nu x)\}_{\nu>0}$ through its limit; the most desirable case is when $\lim\limits_{\nu\to 0^+}G_\R\cdot(\nu x)$ coincides with the closure of $\mathcal{O}$.

In Section \ref{cones}, we relate the limit set $\lim\limits_{\nu\to 0^+}G_\mathbb{R}\cdot(\nu x)$ with various asymptotic cones in ${\mathfrak g}_{\mathbb R}$ or in its complexification ${\mathfrak g}$. 
Namely, consider the sphere 
\begin{equation*}
{\mathbb S}({\mathfrak g}_{\mathbb R}):=({\mathfrak g}_{\mathbb R}\setminus\{0\})/{\mathbb R}^\star_+
\end{equation*}
and the embedding 
\begin{equation*}
\iota^+:{\mathfrak g}_{\mathbb R}\hookrightarrow {\mathbb S}({\mathfrak g}_{\mathbb R}\oplus{\mathbb R}),\;v\mapsto [v\oplus 1].
\end{equation*}
We define the real asymptotic cone ${\mathcal AC}_{\mathbb R}(G_{\mathbb R}\cdot x)$ of $G_{\mathbb R}\cdot x$ as follows (Definition \ref{realAC}):
\begin{equation*}
{\mathcal AC}_{\mathbb R}(G_{\mathbb R}\cdot x):=\{0\}\cup(\pi^+)^{-1}({\mathbb S}({\mathfrak g}_{\mathbb R})\cap\overline{\iota^+(G_{\mathbb R}\cdot x)}),
\end{equation*}
where $\pi^+:{\mathfrak g}_{\mathbb R}\setminus\{0\}\longrightarrow ({\mathfrak g}_{\mathbb R}\setminus\{0\})/{\mathbb R}^\star_+$ is the natural surjection. We observe in  Proposition \ref{P-cones} that our asymptotic cone coincides with the cone introduced by Harris--He--\'Olafsson in \cite{HHO}. 
Then we show that this cone coincides with the limit set 
$\lim\limits_{\nu\to 0^+}G_\mathbb{R}\cdot(\nu x)$ (Theorem \ref{posconethm}):
\begin{equation*}
\lim_{\nu\to 0^+}G_\mathbb{R}\cdot(\nu x)={\mathcal AC}_{\mathbb R}(G_{\mathbb R}\cdot x).
\end{equation*}

In Section \ref{nontriviality}, we deduce that the limit $\lim\limits_{\nu\to0^+}G_\R\cdot(\nu x)$ is always nontrivial whenever $x\not=0$ (Theorem \ref{theorem-nontriv}).
Recall that $G_\R$ is assumed to be simple and non-compact.

{}From the formulation in terms of the asymptotic cone, we also deduce that, when $x$ is semisimple, ${\lim\limits_{\nu\to 0^+}G_\mathbb{R}\cdot(\nu x)}$ is contained in the Zariski closure of the Richardson orbit of $x$ (Theorem \ref{posconethm}): 
\begin{equation*}
\lim_{\nu\to 0^+}G_\mathbb{R}\cdot(\nu x)
\subset \overline{\mathcal{O}_{\mathrm{Rich}}(x)}\cap\mathfrak{g}_\mathbb{R}.
\end{equation*}
Recall that $\mathcal{O}_{\mathrm{Rich}}(x)$ is the unique dense orbit in $G\cdot{\mathfrak u}(x)$, where $G$ is the complexification of $G_{\mathbb R}$ 
with Lie algebra ${\mathfrak g}$ and ${\mathfrak u}(x)$ is the nilradical of any parabolic subalgebra of ${\mathfrak g}$ which contains the centralizer of $x$ as a Levi factor. If, moreover, $x$ is hyperbolic, i.e., $\text{ad}(x)$ has real eigenvalues on ${\mathfrak g}_{\mathbb R}$, and thus on ${\mathfrak g}$, 
the above inclusion becomes an equality in the case of $\mathrm{SL}_n({\mathbb R})$ (Corollary \ref{corollary-SLn}). Writing ${\mathfrak u}_{\mathbb R}(x)$ 
for the nilradical of the real parabolic subalgebra of ${\mathfrak g}_{\mathbb R}$ defined by the positive eigenspaces of $\text{ad}(x)$ we obtain that (Theorem \ref{T4}):
\begin{equation*}
\lim_{\nu\to 0^+}G_\mathbb{R}\cdot(\nu x)=G_\mathbb{R}\cdot\mathfrak{u}_\mathbb{R}(x)=\overline{\mathcal{O}_1}\cup\ldots\cup\overline{\mathcal{O}_k},
\end{equation*}
where $\mathcal{O}_1,\ldots,\mathcal{O}_k$ are nilpotent orbits which are all of the same dimension, for an arbitrary non-compact connected simple linear real Lie group $G_{\mathbb R}$ and a semisimple hyperbolic $x\in\fg_{\mathbb R}$.

Hyperbolic semisimple elements naturally arise in the study of $\fs\fl_2$-triples.
More precisely, by Jacobson--Morozov theorem, every nilpotent element $e\in{\mathfrak g}_{\mathbb R}$ lies in an 
$\fs\fl_2$-triple $\{h,e,f\}$ for some semisimple element $h$ and a nilpotent element $f$ in ${\mathfrak g}_{\mathbb R}$ which are unique up to conjugation under $G_{\mathbb R}$. We show that if the Richardson orbit $\mathcal{O}_{\mathrm{Rich}}(h)$ has real forms which all intersect ${\mathfrak u}_{\mathbb R}(h)$, then 
(Proposition \ref{PropReal}):
\begin{equation*}
\lim_{\nu\to0^+}G_\R\cdot(\nu h)=\overline{\mathcal{O}_{\mathrm{Rich}}(h)\cap{\mathfrak g}_{\mathbb R}}=\overline{\mathcal O_1}\cup\ldots\cup\overline{\mathcal O_r},
\end{equation*}
where ${\mathcal O}_1,\ldots,{\mathcal O}_r$ are the real forms of $\mathcal{O}_{\mathrm{Rich}}(h)$. In particular, this equality holds for any $e\in{\mathfrak g}_{\mathbb R}$ when $G_{\mathbb R}=\mathrm{SL}_n({\mathbb R})$ (see Theorem \ref{noneventhm2}, which includes an explicit description of the limit set), and for an arbitrary non-compact connected simple linear real Lie group $G_{\mathbb R}$ when $e$ is even, i.e., when the eigenvalues of $\text{ad}(h)$ are all even integers (Theorem \ref{T-P3}). 

Finally, we consider the complementary case where $x\in{\mathfrak g}_{\mathbb R}$ is semisimple elliptic, i.e., the eigenvalues of $\ad(x)$ are all pure imaginary.
Suppose that $G_{\mathbb R}$ contains a compact Cartan subgroup $T_{\mathbb R}$. Let $K_{\mathbb R}$ be a maximal compact subgroup of $G_{\mathbb R}$ with complexified Lie algebra ${\mathfrak k}$ and let ${\mathfrak g}={\mathfrak k}\oplus{\mathfrak p}$ be the corresponding Cartan decomposition of ${\mathfrak g}$. Let $\fu=\fu(ix)\subset\fg$ be the nilradical of the parabolic subalgebra associated to the semisimple element $ix\in\fg$. There is a unique $K$-orbit ${\mathcal O}^+$ in ${\mathfrak p}$ which intersects the subspace $\fu\cap{\mathfrak p}$ along a dense open subset. As a consequence of deep results in Representation Theory, it turns out that (see (\ref{formula-elliptic})): 
\begin{equation*}
\lim_{\nu\to0^+}G_\R\cdot(\nu x)=\overline{\KS^{-1}(\mathcal{O}^+)},
\end{equation*}
where $\text{KS}$ is the Kostant--Sekiguchi bijection between nilpotent adjoint orbits of $G_{\mathbb R}$ and nilpotent $K$-orbits of $\fp$. To our best knowledge, there is no direct proof of this equality. Based on this equality, we prove that every even nilpotent orbit of $\fg_\R$ can be approximated by a continuous family of semisimple elliptic orbits.
Specifically,
if $\{h,e,f\}$ is an even $\mathfrak{sl}_2$-triple, then we have
(Theorem \ref{T-elliptic-even}):
\[\lim_{\nu\to 0^+}G_\R\cdot(\nu(e-f))=\overline{G_\R\cdot e}.\]
In the case where $G_\R$ is classical, we provide continuous families of semisimple elliptic orbits whose limits attempt to approximate minimal nilpotent orbits (Theorem \ref{T-minimal}).
Finally, for $G_\R=\mathrm{SU}(p,q)$, we show that (the closure of) every nilpotent orbit can be realized as the limit of a family of semisimple elliptic orbits (Theorem \ref{T-SUpq}).

\section{Preliminaries}\label{preliminaries}

Let $G_{\mathbb R}$ be a connected simple linear real Lie group with Lie algebra 
$\fg_{\mathbb R}$ and let $G$ be the complexification of $G_{\mathbb R}$ whose Lie algebra ${\mathfrak g}$ is the complexification of $\fg_{\mathbb R}$. 
Fix a maximal compact subgroup $K_{\mathbb R}$ 
of $G_{\mathbb R}$ with Lie algebra $\fk_{\mathbb R}$ fixed pointwise by a Cartan involution $\theta$. The Lie algebra $\fk$ of the complexification $K$ of $K_{\mathbb R}$ is the complexification of $\fk_{\mathbb R}$. Write 
\begin{equation}\label{cartandecomp}
\fg_{\mathbb R}=\fk_{\mathbb R}\oplus\fp_{\mathbb R}\;\text{ and }\;\fg=\fk\oplus\fp
\end{equation}
for the corresponding Cartan decompositions of $\fg_{\mathbb R}$ and $\fg$, respectively. 

An element $x$ of $\fg_{\mathbb R}$ (resp., $\fg$) is said to be {\it semisimple} if the endomorphism $\ad(x):\fg_{\mathbb R}\rightarrow \fg_{\mathbb R}$ (resp., $\ad(x):\fg\rightarrow \fg$) is diagonalizable. A semisimple element $x$ of $\fg_{\mathbb R}$ is said to be {\it elliptic} (resp., {\it hyperbolic}) if the eigenvalues of $\ad(x)$ are all pure imaginary (resp., real). An element $x$ of $\fg_{\mathbb R}$ (resp., $\fg$) is said to be {\it nilpotent} if the endomorphism $\ad(x):\fg_{\mathbb R}\rightarrow \fg_{\mathbb R}$ (resp., $\ad(x):\fg\rightarrow \fg$) is nilpotent. By Jordan decomposition, any element $x$ of $\fg_{\mathbb R}$ (resp., $\fg$) can be written in a unique way as 
\begin{equation*}
x=x_s+x_n,
\end{equation*}
where $x_s$ (resp., $x_n$) is semisimple (resp., nilpotent) in $\fg_{\mathbb R}$ (resp., $\fg$) and $[x_s,x_n]=0$. Moreover, any element in $\fg$ commuting with $x$ commutes with $x_s$ and $x_n$ as well.

If $x$ belongs to $\fg_{\mathbb R}$, then its semisimple part $x_s$ can be written in a unique way as
\[
x_s=x_e+x_h,\quad\mbox{that is,}\quad x=x_e+x_h+x_n,
\]
where $x_e$ is elliptic, $x_h$ is hyperbolic, and $x_e,x_h,x_n$ commute with each other. 
Moreover, any $G_\mathbb{R}$-orbit in $\fg_\mathbb{R}$ contains an element $x$ such that $x_e$ belongs to $\fk_\mathbb{R}$ and $x_h$ belongs to $\fp_\mathbb{R}$
(see \cite[Proposition 2.10]{Vogan}).

If $e$ is a nonzero nilpotent element of $\fg_{\mathbb R}$ (resp., $\fg$), by Jacobson--Morozov theorem there exist a semisimple element $h$ and a nilpotent element $f$ in $\fg_{\mathbb R}$ (resp., $\fg$) such that (\cite[Theorem 9.2.1]{CM}):
\begin{equation*}
[h,e]=2e,\;\;[h,f]=-2f,\;\text{ and }\;[e,f]=h.
\end{equation*}
The triple $\{h,e,f\}$ is said to be a {\it standard triple} (or {\it $\mathfrak{sl}_2$-triple}), while $h$ and $e$ are respectively the {\it neutral element} and the {\it nilpositive element} of the triple. It is known that any standard triple $\{h,e,f\}\subset\mathfrak{g}_\mathbb{R}$ is $G_{\mathbb R}$-conjugate to a {\it Cayley triple} $\{h^\prime,e^\prime,f^\prime\}$, i.e., such that $\theta(e^\prime)=-f^\prime$, $\theta(f^\prime)=-e^\prime$, and  $\theta(h^\prime)=-h^\prime$. Furthermore, we associate with a Cayley triple $\{h,e,f\}\subset\mathfrak{g}_{\mathbb{R}}$ the standard triple $\{h^\prime,e^\prime,f^\prime\}\subset\mathfrak{g}$ given by 
\begin{eqnarray*}
h^\prime&=&i(e-f),\\
e^\prime&=&\frac{1}{2}(e+f+ih),\\
f^\prime&=&\frac{1}{2}(e+f-ih).
\end{eqnarray*}
The triple $\{h^\prime,e^\prime,f^\prime\}$ is the {\it Cayley transform} of $\{h,e,f\}$. Note that $h^\prime$ lies in $\fk$, while $e^\prime$ and $f^\prime$ lie in $\fp$. Such a standard triple is called {\it normal}. The Kostant--Sekiguchi correspondence is the bijection
\begin{equation}\label{KS}
\KS: G_{\mathbb R}\cdot e\mapsto K\cdot e^\prime
\end{equation}
between nilpotent $G_{\mathbb R}$-orbits of $\fg_{\mathbb R}$ and nilpotent $K$-orbits of $\fp$ (\cite[Chapter 9]{CM}). We will use the following result due to Mal'cev about conjugation of standard triples. 
\begin{thm}[Mal'cev, {\cite[\S3.4.12]{CM}}]
\label{malcev} 
Any two standard triples of $\mathfrak{g}$ with the same neutral element $h$ are conjugate by an element of the connected component of the centralizer $Z_G(h)$ of $h$ in $G$.
\end{thm}

Fix a standard triple $\{h,e,f\}$ in $\fg$. 
The eigenvalues of $\ad(h)$ are integers, therefore one gets the following grading of $\fg$:
\begin{equation*}
\fg=\bigoplus_{m\in\mathbb{Z}}\fg_m,\quad \fg_m:=\{X\in\fg\mid [h,X]=mX\}.
\end{equation*}
Set $\fl=\fg_0$ and $\fu=\bigoplus\limits_{m>0}\fg_m$. The Lie subalgebra 
\begin{equation}\label{parabolic}
\fq=\fl\oplus\fu
\end{equation}
is a {\it parabolic subalgebra} associated with the nilpotent element $e$. 
Note that $\mathfrak{q}$ is $\theta$-stable when $\{h,e,f\}$ is a Cayley triple in $\mathfrak{g}_\mathbb{R}$.
The corresponding parabolic subgroup $Q$ in $G$ has Levi decomposition $Q=LU$, where $L$ (resp., $U$) has Lie algebra ${\mathfrak l}$ (resp., ${\mathfrak u}$). In the case where $\fg_{2j+1}=\{0\}$ for all $j\in{\mathbb Z}$, the nilpotent element $e$ (or the standard triple $\{h,e,f\}$) is said to be {\it even}. In other words, $\fg$ decomposes as the direct sum of irreducible representations of the subalgebra spanned by the ${\fs\fl}_2$-triple $\{h,e,f\}$ and all of the summands have even highest weight. In the special case where $\fl$ is a maximal torus $\ft$ in $\fg$, the nilpotent radical $\fu$ coincides with the subspace $\fn$ generated by positive $\ft$-roots in $\fg$, and the parabolic subgroup $Q$ is a Borel subgroup $B=LN$ with Lie algebra $\fb=\ft\oplus\fn$ and $N=\exp(\fn)$.

By assumption, $G$ is a closed subgroup of $\GLnC$ for some integer $n\geq 1$. The nilpotent cone ${\mathcal N}(\glnC)$ is defined as the set of nilpotent elements in $\text{End}({\mathbb C}^n)$. Similarly, the nilpotent cones ${\mathcal N}(\fg)$ and ${\mathcal N}(\fg_{\mathbb R})$ in $\fg$ and $\fg_{\mathbb R}$ are defined, respectively, as the sets of nilpotent elements in $\fg$ and in $\fg_{\mathbb R}$. We have 
$${\mathcal N}(\fg)=\fg\cap {\mathcal N}(\glnC).$$
Since ${\mathcal N}(\glnC)$ is Zariski closed in $\glnC$, then the nilpotent cone ${\mathcal N}(\fg)$ is Zariski closed in $\fg$. On the other hand, the algebra $\mathcal{S}(\fg^*)^G$ of $G$-invariant polynomial functions on $\fg$ is graded by the degree: 
\begin{equation*}
\mathcal{S}(\fg^*)^G=\bigoplus_{r\geq 0} S^r(\fg^*)^G,
\end{equation*}
where $S^r(\fg^*)$ is the space of homogeneous polynomial functions of degree $r$ on $\fg$. 
If $S^+(\fg^*)^G:=\bigoplus\limits_{r>0} S^r(\fg^*)^G$, then 
 \begin{equation*}
{\mathcal N}(\fg)=\big\{x\in\fg\mid f(x)=0\text{ for all }f\in S^+(\fg^*)^G\big\}=G\cdot\fn.
\end{equation*}
In particular, the variety ${\mathcal N}(\fg)$ is irreducible and 
\begin{equation*}
\dim{\mathcal N}(\fg)=2\dim\fn=\dim\fg-\rank\,\fg.
\end{equation*}
Here $\rank\,\fg$ denotes the rank of $\fg$, i.e., the dimension of a Cartan subalgebra in $\fg$. The nilpotent cone ${\mathcal N}(\fg)$ is a finite union of nilpotent $G$-orbits. There are several nilpotent orbits of particular interest in ${\mathcal N}(\fg)$ \cite[Chapters 4 and 7]{CM}. First, the nilpotent cone coincides with the Zariski closure of a unique nilpotent $G$-orbit ${\mathcal O}_{\mathrm{reg}}$ which is open and dense in ${\mathcal N}(\fg)$:
\begin{equation*}
{\mathcal N}(\fg)=\overline{{\mathcal O}_{\mathrm{reg}}}.
\end{equation*}
It is known as the {\it regular} or {\it principal nilpotent orbit}, and it is the largest nilpotent $G$-orbit in ${\mathcal N}(\fg)$. Elements in ${\mathcal O}_{\mathrm{reg}}$ will be called {\it regular nilpotent}. Second, since $\fg$ is simple, there exists a unique nilpotent $G$-orbit ${\mathcal O}_{\mathrm{subreg}}$ of dimension $\dim {\mathcal O}_{\mathrm{reg}}-2=\dim\fg-\rank\,\fg-2$ in ${\mathcal N}(\fg)$. The orbit ${\mathcal O}_{\mathrm{subreg}}$, known as the {\it subregular nilpotent orbit}, is open and dense in 
${\mathcal N}(\fg)\backslash{\mathcal O}_{\mathrm{reg}}$. Third, there exists a nonzero nilpotent $G$-orbit ${\mathcal O}_{\mathrm{min}}$ in ${\mathcal N}(\fg)$ which is contained in the Zariski closure of any nonzero nilpotent orbit. The orbit ${\mathcal O}_{\mathrm{min}}$ has minimal dimension and is known as the {\it minimal nilpotent orbit}. 

Yet there is another particular nilpotent orbit in ${\mathfrak g}$, induced from any parabolic subalgebra ${\mathfrak q}$ in $\mathfrak{g}$. Namely, the {\it Richardson orbit} ${\mathcal O}_{\mathrm{Rich}}({\mathfrak q})$ is the unique dense orbit in $G\cdot {\mathfrak u}$, where ${\mathfrak u}$ denotes the nilradical of $\mathfrak{q}$. Note that any semisimple element $x\in\mathfrak{g}$ also gives rise to a Richardson orbit.
Indeed, if $x$ is semisimple, then its centralizer $\mathfrak{g}^x:=\{z\in\mathfrak{g}\mid [x,z]=0\}$ 
is a Levi subalgebra of $\mathfrak{g}$, which means that there is a parabolic subalgebra $\mathfrak{q}$
such that $\mathfrak{g}^x$ is a Levi subalgebra of $\mathfrak{q}$.
The Richardson orbit 
\begin{equation}\label{DefRich}
{\mathcal O}_{\mathrm{Rich}}(x):={\mathcal O}_{\mathrm{Rich}}(\mathfrak{q})
\end{equation}
is independent of the choice of $\mathfrak{q}$ (see, e.g., \cite[Theorem 7.1.3]{CM}). Note also that
\begin{equation}
\label{dim-Rich}
\dim{\mathcal O}_{\mathrm{Rich}}(x)=2\dim\mathfrak{u}=\dim\mathfrak{g}-\dim\mathfrak{g}^x=\dim G\cdot x
\end{equation}
(see \cite[\S7]{CM} for the first equality).
In the case where $x$ has real eigenvalues on $\mathfrak{g}$
(which occurs for instance when $x$ is a hyperbolic semisimple element in $\mathfrak{g}_\mathbb{R}$),
then $\mathfrak{g}$ has an eigenspace decomposition indexed by real numbers
\begin{equation}
\label{grading-hyp}
\mathfrak{g}=\bigoplus_{\mu\in\mathbb{R}}\mathfrak{g}_\mu,\quad\mathfrak{g}_\mu:=\{z\in \mathfrak{g}\mid [x,z]=\mu z\},
\end{equation}
and $\mathfrak{q}:=\bigoplus\limits_{\mu\geq 0}\mathfrak{g}_\mu$ is a parabolic subalgebra which contains $\mathfrak{g}_0=\mathfrak{g}^x$ as a Levi factor; we have
\begin{equation}
\label{Richardson-hyp}
\overline{\mathcal{O}_{\mathrm{Rich}}(x)}=G\cdot \Big(\bigoplus_{\mu>0}\mathfrak{g}_\mu\Big).
\end{equation}
If $x=h$ belongs to an $\mathfrak{sl}_2$-triple $\{h,e,f\}$, then the nilpotent orbit
$G\cdot e$ is contained in the closure of ${\mathcal O}_{\mathrm{Rich}}(h)$.
The equality $\mathcal{O}_{\mathrm{Rich}}(h)=G\cdot e$ holds if and only if $e$ is even.
Any even nilpotent orbit is therefore a Richardson orbit, but the converse is not true. 
Since ${\mathfrak g}$ is simple, the subregular orbit ${\mathcal O}_{\mathrm{subreg}}$ is a Richardson orbit. 
Note that every nilpotent orbit of $\slnC$ is a Richardson orbit. 

Recall that nilpotent $G$-orbits in the classical Lie algebras are in one-to-one correspondence 
with partitions $(d_1,\ldots,d_k)$ with $d_1\geq d_2\geq\ldots\geq
d_k\geq 1$ such that (\cite[Chapter~5]{CM}): 
\begin{itemize}
\item[$\bullet$] $d_{1}+d_{2}+\ldots+d_{k}=n$, when
  ${\mathfrak g}\simeq\slnC$; 
\item[$\bullet$] $d_{1}+d_{2}+\ldots+d_{k}=2n+1$ and the even $d_j$'s
  occur with even multiplicity, when ${\mathfrak g}\simeq\soC{2n+1}$; 
\item[$\bullet$] $d_{1}+d_{2}+\ldots+d_{k}=2n$ and the odd $d_j$'s occur
  with even multiplicity, when ${\mathfrak g}\simeq\spC{2n}$; 
\item[$\bullet$] $d_{1}+d_{2}+\ldots+d_{k}=2n$ and the even $d_j$'s
  occur with even multiplicity, when ${\mathfrak g}\simeq\soC{2n}$;
  except that the partitions having all the
   $d_j$'s even (and occurring with even multiplicity) are each 
  associated to {\em two} orbits. 
\end{itemize}
Finally, the Zariski closure of nilpotent orbits can be described as follows. Given two partitions 
${\bf d}=(d_1,\ldots,d_k)$ and ${\bf f}=(f_1,\ldots,f_\ell)$, ${\bf d}$ is said to 
{\it dominate} ${\bf f}$, i.e., ${\bf f}\leq {\bf d}$, if we have $k\leq\ell$ and
\begin{equation*}
\sum_{1\leq i\leq j}d_i\geq\sum_{1\leq i\leq j}f_i\;\text{ for }1\leq j\leq k.
\end{equation*}
If $\mathcal{O}$ and $\mathcal{O}'$ are nilpotent orbits associated with distinct partitions $\mathbf{d}$ and $\mathbf{d}'$, then $\mathcal{O}$ is contained in the Zariski closure of $\mathcal{O}'$ if and only if $\mathbf{d}\leq\mathbf{d}'$.

\section{Limit of adjoint orbits: definition and basic properties}\label{limitorbit}

In this section we introduce a topological limit of adjoint orbits (Definition \ref{deflim}) and establish its basic properties (Propositions \ref{basicprop}, \ref{P3.5}, and \ref{P-existence}). Recall that $G_\mathbb{R}$ is a connected simple linear real Lie group with Lie algebra $\mathfrak{g}_{\mathbb R}$, $G$ is the complexification of $G_{\mathbb R}$, and its Lie algebra $\mathfrak{g}$ is the complexification of $\mathfrak{g}_\mathbb{R}$. 

\begin{Def}\label{deflim}
Given an element $x\in\mathfrak{g}_\mathbb{R}$, consider the sequence $\{G_\mathbb{R}\cdot(\nu x)\}_{\nu>0}$ of $G_{\mathbb R}$-orbits. The limit of this sequence of orbits is defined as the topological set
$$
\lim_{\nu\to 0^+}G_\mathbb{R}\cdot(\nu x):=\bigcap_{\epsilon>0}\overline{\bigcup_{\nu\in(0,\epsilon)}G_\mathbb{R}\cdot(\nu x)}.
$$
\end{Def}

In other words, an element $z$ belongs to the limit if and only if there are sequences $\{\nu_k\}_{k\geq 1}\subset\mathbb{R}_+^*$ converging to $0$ and $\{g_k\}_{k\geq 1}\subset G_\mathbb{R}$ such that 
\[z=\lim_{k\to+\infty} \nu_k\mathrm{Ad}(g_k)(x).\]

\begin{pro}\label{basicprop}
Let $x\in{\mathfrak g}_{\mathbb R}$ with Jordan decomposition $x=x_s+x_n$.
\begin{itemize}
\item[(a)] The limit of orbits $\lim\limits_{\nu\to 0^+}G_\mathbb{R}\cdot(\nu x)$ is nonempty, closed, $G_{\mathbb R}$-stable and contained in the nilpotent cone  of ${\mathfrak g}_{\mathbb R}$:
$$\lim_{\nu\to 0^+}G_\mathbb{R}\cdot(\nu x)\subset {\mathcal N}(\fg_{\mathbb R}).$$
In particular, the limit of orbits $\lim\limits_{\nu\to 0^+}G_\mathbb{R}\cdot(\nu x)$ is a finite union of nilpotent orbits.
\item[(b)] The nilpotent part $x_n$ belongs to the limit of orbits $\lim\limits_{\nu\to 0^+}G_\mathbb{R}\cdot(\nu x)$. In particular, $\lim\limits_{\nu\to 0^+}G_\mathbb{R}\cdot(\nu x)$ always contains the closure of $G_{\mathbb R}\cdot x_n$:
$$\overline{G_{\mathbb R}\cdot x_n}\subset\lim_{\nu\to 0^+}G_\mathbb{R}\cdot(\nu x).$$
\item[(c)] The limit of orbits of the semisimple part $x_s$ is contained in the limit of orbits of $x$:
$$
\lim_{\nu\to 0^+}G_\mathbb{R}\cdot(\nu x_s)\subset \lim_{\nu\to 0^+}G_\mathbb{R}\cdot(\nu x).
$$
\item[(d)] If $x=x_n$ is a nilpotent element then the limit of orbits $\lim\limits_{\nu\to 0^+}G_\mathbb{R}\cdot(\nu x)$ coincides with the closure of the $G_{\mathbb R}$-orbit of $x$:
$$\lim_{\nu\to 0^+}G_\mathbb{R}\cdot(\nu x)=\overline{G_{\mathbb R}\cdot x}.$$
\end{itemize}
\end{pro}
\begin{proof} (a) By definition $\lim\limits_{\nu\to 0^+}G_{\mathbb R}\cdot(\nu x)$ is closed and $G_{\mathbb R}$-stable,
and nonempty (since it contains $0$). Let $z\in\lim\limits_{\nu\to 0^+}G_\mathbb{R}\cdot(\nu x)$ and let us show that $z$ is nilpotent in ${\mathfrak g}_{\mathbb R}$. 
There are sequences $\{\nu_k\}_{k\geq 1}\subset\mathbb{R}_+^*$ converging to $0$ and $\{g_k\}_{k\geq 1}\subset G_\mathbb{R}$
such that 
\[
z=\lim_{k\to+\infty}\nu_k \mathrm{Ad}(g_k)(x),
\]
hence
\[
\mathrm{ad}(z)=\lim_{k\to+\infty}\nu_k \mathrm{Ad}(g_k)\circ\mathrm{ad}(x)\circ\mathrm{Ad}(g_k)^{-1}.
\]
This yields the following relation between characteristic polynomials:
\[
\det(\mathrm{ad}(z)-X\mathrm{id})=\lim_{k\to+\infty}\det(\nu_k\mathrm{ad}(x)-X\mathrm{id}),
\]
from which we conclude that the endomorphism $\mathrm{ad}(z)$ is nilpotent, hence $z$ is a nilpotent element.
This shows (a).

To show (b), (c), and (d), it is useful to note that
\begin{equation} \label{jordanorbit}
G_\mathbb{R}\cdot(x_s+tx_n)=G_\mathbb{R}\cdot(x_s+x_n)\quad\mbox{for all $t\in\mathbb{R}_+^*$.}
\end{equation}
Indeed, the Jacobson--Morozov theorem implies that we can find an element $x'$ in the semisimple part $[\mathfrak{l},\mathfrak{l}]$
of the Levi subalgebra $\mathfrak{l}=\mathfrak{z}_{\mathfrak{g}_\mathbb{R}}(x_s):=\{z\in\mathfrak{g}_\mathbb{R}\mid [x_s,z]=0\}$ such that $[x',x_n]=2x_n$. Then, for every $t>0$ we get
\[
\textstyle
x_s+tx_n=\mathrm{Ad}(\exp(\frac{\ln t}{2}x'))(x_s+x_n)\in G_\mathbb{R}\cdot(x_s+x_n).
\]

(b) By (\ref{jordanorbit}), we have $x_n+\nu x_s\in G_\mathbb{R}\cdot (\nu x)$ for all $\nu>0$,
hence $x_n\in\lim\limits_{\nu\to 0^+}G_\mathbb{R}\cdot(\nu x)$, and the desired inclusion
follows from the properties of the limit given in (a).

(c) By (\ref{jordanorbit}), for every $\nu>0$, we have
\[
\nu x_s+\frac{1}{k} x_n\in G_\mathbb{R}\cdot (\nu x)\subset\overline{G_\mathbb{R}\cdot(\nu x)}\quad\mbox{for all $k\geq 1$,}
\]
hence
\[
G_\mathbb{R}\cdot(\nu x_s)\subset \overline{G_\mathbb{R}\cdot(\nu x)},
\]
so that the claimed inclusion follows from the definition of the limit.

(d) By (\ref{jordanorbit}), we have $G_\mathbb{R}\cdot x=G_\mathbb{R}\cdot(\nu x)$ for all $\nu>0$, and
the claimed equality immediately follows from the definition of the limit.
\end{proof}

\begin{ex}
\label{E1}
The elements of the Lie algebra $\mathfrak{g}_\mathbb{R}=\slR{2}$ can be written in the form
\[x=\begin{pmatrix} x_1 & x_2+x_3 \\ x_2-x_3 & -x_1\end{pmatrix}\quad\mbox{with $x_1,x_2,x_3\in\mathbb{R}$}.\]
The nilpotent cone is
\[
\mathcal{N}(\slR{2})
=
\{x\in\slR{2}\mid \det x=0\}=
\{x\in\slR{2}\mid x_1^2+x_2^2=x_3^2\}
\]
and it comprises two nontrivial nilpotent orbits,
\begin{eqnarray*}
\mathfrak{O}^+&:=&\{x\in\slR{2}\mid x_1^2+x_2^2=x_3^2,\ x_3>0\}=\mathrm{SL}_2(\mathbb{R})\cdot\begin{pmatrix} 0 & 1 \\ 0 & 0 \end{pmatrix},\\
\mathfrak{O}^-&:=&\{x\in\slR{2}\mid x_1^2+x_2^2=x_3^2,\ x_3<0\}=\mathrm{SL}_2(\mathbb{R})\cdot\begin{pmatrix} 0 & 0 \\ 1 & 0 \end{pmatrix}.
\end{eqnarray*}

We consider the elliptic semisimple element
\[
e:=\begin{pmatrix}
0 & 1 \\ -1 & 0
\end{pmatrix}.
\]
In fact, every elliptic semisimple element of $\slR{2}$ is conjugate to $\nu e$ for some $\nu\in\mathbb{R}$.
For $\nu\in\mathbb{R}^*$, we have 
\[
\mathrm{SL}_2(\mathbb{R})\cdot (\nu e)=\{x\in\slR{2}\mid x_3^2=x_1^2+x_2^2+\nu^2,\ \mathrm{sign}(x_3)=\mathrm{sign}(\nu)\}.
\]
This yields
\begin{eqnarray*}
\lim_{\nu\to 0^+}\mathrm{SL}_2(\mathbb{R})\cdot(\nu e) & = & 
\bigcap_{\epsilon>0}\overline{\bigcup_{0<\nu<\epsilon}\mathrm{SL}_2(\mathbb{R})\cdot (\nu e)} \\
 & = & 
\bigcap_{\epsilon>0}\{x\in\slR{2}\mid 0\leq x_3^2-x_1^2-x_2^2\leq \epsilon^2,\ x_3\geq 0\} \\
 & = & 
\{x\in\slR{2}\mid x_3^2=x_1^2+x_2^2,\ x_3\geq 0\} \\
 & = & \{0\}\cup\mathfrak{O}^+ = \overline{\mathfrak{O}^+}.
\end{eqnarray*}
A similar calculation shows that 
\[\lim_{\nu\to 0^+}\mathrm{SL}_2(\mathbb{R})\cdot(\nu (-e))=\{0\}\cup\mathfrak{O}^-=\overline{\mathfrak{O}^-}.\]
We now consider the hyperbolic semisimple element
\[
h:=\begin{pmatrix}
1 & 0 \\ 0 & -1
\end{pmatrix}
\]
and let us note that every hyperbolic semisimple element of $\slR{2}$ is conjugate to $\nu h$ for some $\nu\geq 0$; in particular $h$ and $-h$ are conjugate under $\mathrm{SL}_2(\mathbb{R})$.
For $\nu>0$, we have
\[
\mathrm{SL}_2(\mathbb{R})\cdot(\nu h)=\{x\in\slR{2}\mid x_3^2=x_1^2+x_2^2-\nu^2\},
\]
so that
\[
\lim_{\nu\to 0^+}\mathrm{SL}_2(\mathbb{R})\cdot(\nu h)=\{0\}\cup\mathfrak{O}^+\cup\mathfrak{O}^-=\mathcal{N}(\slR{2}).
\]
This example is illustrated in Figure \ref{F1}.
\begin{figure}[H]
\includegraphics[scale=0.8]{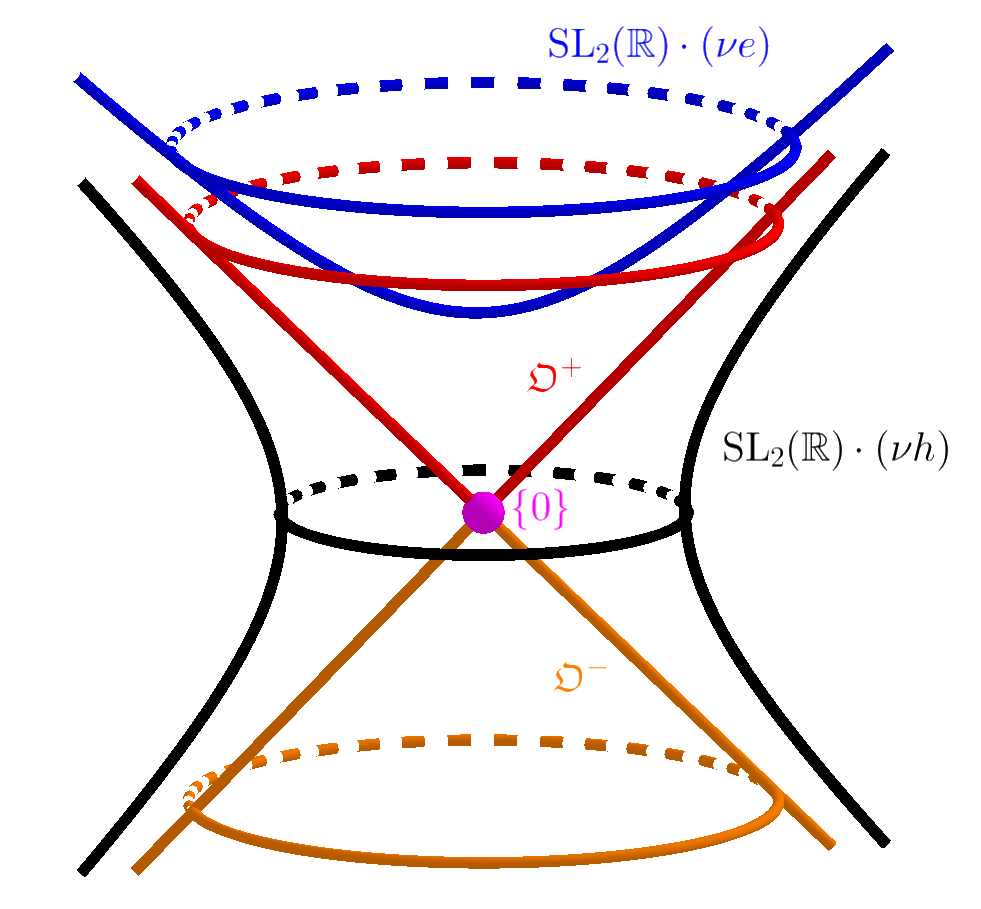}
\caption{The nilpotent cone $\mathcal{N}(\slR{2})$ and the orbits $\mathrm{SL}_2(\mathbb{R})\cdot(\nu e)$ and $\mathrm{SL}_2(\mathbb{R})\cdot(\nu h)$ for $\nu=\sqrt{3}$.}
\label{F1}
\end{figure}
\end{ex}

\begin{rem}
\label{R3.5}
The mapping $\phi:\fg_\mathbb{R}\to\fg_\mathbb{R}$, $x\mapsto -x$ is a $G_\mathbb{R}$-equivariant linear isomorphism, hence it induces a homeomorphism
\[
\lim_{\nu\to 0^+}G_\mathbb{R}\cdot(\nu x)\stackrel{\sim}{\to}\lim_{\nu\to 0^+}G_\mathbb{R}\cdot(\nu (-x)).
\]
Moreover, $\phi$ induces an involution
\[\tilde\phi:\mathcal{N}(\fg_\mathbb{R})/G_\mathbb{R}\to \mathcal{N}(\fg_\mathbb{R})/G_\mathbb{R}\]
of the set of nilpotent orbits, and we get
\[
\lim_{\nu\to 0^+}G_\mathbb{R}\cdot(\nu (-x))=\bigcup \tilde\phi(\mathcal{O}),
\]
where the union is over the set of nilpotent orbits $\mathcal{O}$ contained in $\lim\limits_{\nu\to 0^+}G_\mathbb{R}\cdot(\nu x)$.
For instance, in Example \ref{E1} above, we have $\tilde\phi(\mathfrak{O}^+)=\mathfrak{O}^-$ and of course $\tilde\phi(\{0\})=\{0\}$. In this way the formula for $\lim\limits_{\nu\to 0^+}\mathrm{SL}_2(\mathbb{R})\cdot(\nu(-e))$
can also be deduced from the formula giving $\lim\limits_{\nu\to 0^+}\mathrm{SL}_2(\mathbb{R})\cdot(\nu e)$.
\end{rem}

The following is a characterization of the limit of orbits in terms of the Slodowy slice of a nilpotent element. This characterization appears in \cite[Corollary 2.2]{Harris} (see also the remark below), where it is attributed to Barbasch and Vogan.
For the sake of completeness, we give a proof.

\begin{pro}
\label{P3.5}
Let $x\in\fg_\R$. Let $e\in\fg_\R$ be a nilpotent element and let $\{h,e,f\}\subset\fg_\R$ be a standard triple. Let $\mathfrak{z}_{\fg_\R}(f)=\{z\in\fg_\R \mid [z,f]=0\}$. The affine space $e+\mathfrak{z}_{\fg_\R}(f)$ is called a Slodowy slice. The following conditions are equivalent:
\begin{itemize}
\item[\rm (i)] $e$ belongs to the limit of orbits $\lim\limits_{\nu\to0^+}G_\R\cdot(\nu x)$;
\item[\rm (ii)] $G_\R\cdot x$ intersects the Slodowy slice $e+\mathfrak{z}_{\fg_\R}(f)$.
\end{itemize}
\end{pro}

\begin{proof}
(ii)$\Rightarrow$(i): We may assume that $x=e+z$ with $z\in\mathfrak{z}_{\fg_\R}(f)$.
It follows from $\mathfrak{sl}_2$-theory that $z$ can be written
as
\[z=\sum_{j=0}^mz_j\quad\mbox{with}\quad [h,z_j]=-jz_j.\] 
Then
\[
\mathrm{Ad}(\exp(th))(x)=\sum_{k\geq 0}\frac{t^k}{k!}\mathrm{ad}(h)^k(e+z)
=\sum_{k\geq 0}\frac{(2t)^k}{k!} e+\sum_{j=0}^m\sum_{k\geq 0}\frac{(-jt)^k}{k!}z_j=\exp(2t)e+\sum_{j=0}^m\exp(-jt)z_{j}
\]
hence, letting $t=-\frac{\ln(\nu)}{2}$, we get
\[
\mathrm{Ad}(\exp(th))(\nu x)=e+\sum_{j=0}^m \nu^{1+\frac{j}{2}}z_j\stackrel{\nu\to 0^+}{\longrightarrow} e.
\]
Therefore, $e\in\lim\limits_{\nu\to 0^+}G_\R\cdot(\nu x)$.

(i)$\Rightarrow$(ii): Assume that $e\in\lim\limits_{\nu\to 0^+}G_\R\cdot(\nu x)$.
This implies that $e\in\overline{\bigcup\limits_{\nu\in(0,1)}G_\R\cdot(\nu x)}$. Note that $G_\R\cdot(e+\mathfrak{z}_{\fg_\R}(f))$ is an open neighborhood of $e$ in $\fg_\R$, hence there are $g\in G_\R$, $\nu\in(0,1)$, and $z\in\mathfrak{z}_{\fg_\R}(f)$ such that 
\[x=\mathrm{Ad}(g)\Big(\frac{1}{\nu}e+z\Big).\]
Letting $t=\frac{\ln\nu}{2}$, we have
\[\mathrm{Ad}(\exp(t h))\Big(\frac{1}{\nu} e\Big)=\frac{1}{\nu}\sum_{k\geq 0}\frac{t^k}{k!}\mathrm{ad}(h)^k(e)=\frac{1}{\nu}\sum_{k\geq 0}\frac{(2t)^k}{k!}e=e\]
while $z':=\mathrm{Ad}(\exp(th))(z)\in\mathfrak{z}_{\fg_\R}(f)$, since $\mathfrak{z}_{\fg_\R}(f)$ is stable by $\mathrm{ad}(h)$. Therefore,
\[
e+z'=(\mathrm{Ad}(\exp(th))\circ\mathrm{Ad}(g^{-1}))(x)\in (G_\R\cdot x)\cap(e+\mathfrak{z}_{\fg_\R}(f)).
\]
This shows (ii).
\end{proof}

\begin{rem}
In \cite[Corollary 2.2]{Harris}, the set 
\[
\mathcal{N}_x:=\mathcal{N}(\fg_\R)\cap\overline{\bigcup_{t>0}G_\R\cdot(tx)}
\]
(or, in fact, its analogue in the dual Lie algebra $\fg_\R^*$) is considered. Note that this set coincides with the limit of orbits:
\begin{equation}\label{3.3-new}
\mathcal{N}_x=\lim_{\nu\to 0^+}G_\R\cdot(\nu x).
\end{equation}
Indeed, the inclusion 
\[\lim_{\nu\to 0^+}G_\R\cdot(\nu x)=\bigcap_{\epsilon>0}\overline{\bigcup_{\nu\in(0,\epsilon)}G_\R\cdot(\nu x)}\subset\overline{\bigcup_{\nu>0}G_\R\cdot(\nu x)}\]
is clear, and since we know that $\lim\limits_{\nu\to 0^+}G_\R\cdot(\nu x)\subset\mathcal{N}(\fg_\R)$ (see Proposition \ref{basicprop}\,(a)), we get the inclusion $\supset$ in (\ref{3.3-new}). For the other inclusion, first note that this inclusion holds in the case where $x$ is nilpotent, because we then have $G_\R\cdot(tx)=G_\R\cdot x$ for all $t>0$, hence 
\[\mathcal{N}_x=\overline{G_\R\cdot x}=\lim_{\nu\to 0^+}G_\R\cdot(\nu x)\]
(see Proposition \ref{basicprop}\,(d)). We now assume that $x$ is not nilpotent.
Let $z\in\mathcal{N}_x$. Then $z$ is nilpotent and there are sequences $\{t_k\}_{k\geq 1}\subset(0,+\infty)$ and $\{g_k\}_{k\geq 1}\subset G_\R$ such that $z=\lim\limits_{k\to+\infty}t_k \mathrm{Ad}(g_k)(x)$.
We claim that the sequence $\{t_k\}_{k\geq 1}$ converges to $0$; in which case we are able to conclude that $z$ belongs to the limit of orbits as explained below Definition \ref{deflim}.
Arguing by contradiction, assume that
along a relabeled subsequence $\{t_k\}_{k\geq 1}$, we have $t_k\to t_0>0$ or 
$t_k\to+\infty$, then this yields $\frac{1}{t_0}z\in\overline{G_\R\cdot x}$ or $0\in\overline{G_\R\cdot x}$, and in both cases we get that $x$ must be nilpotent since $\frac{1}{t_0}z$ and $0$ are nilpotent, a contradiction.

\end{rem}

We conclude this section with the observation that every nilpotent orbit fits into a limit of semisimple orbits. The statement is actually more precise.

\begin{pro} \label{P-existence}
Let $e\in\mathfrak{g}_\mathbb{R}$ be a nilpotent element.
Then there is $x_h\in\mathfrak{g}_\mathbb{R}$ semisimple, hyperbolic
and there is $x_e\in\mathfrak{g}_\mathbb{R}$ semisimple, elliptic
such that
\[
G_\mathbb{R}\cdot e\subset\lim_{\nu\to 0^+}G_\mathbb{R}\cdot(\nu x_h)
\quad\mbox{and}\quad
G_\mathbb{R}\cdot e\subset\lim_{\nu\to 0^+}G_\mathbb{R}\cdot(\nu x_e).
\]
\end{pro}

\begin{proof}
Let $\{h,e,f\}\subset\mathfrak{g}_\mathbb{R}$ be an $\mathfrak{sl}_2$-triple which contains $e$.
Since $\mathrm{ad}(h)$ has integer eigenvalues in $\mathfrak{g}_\mathbb{R}$, the element $x_h:=h$
is hyperbolic. We have
\[
\mathrm{Ad}(\exp({\textstyle -\frac{1}{2\nu}e}))(\nu x_h)=\nu x_h+e\longrightarrow e\quad\mbox{as $\nu\to 0$},
\]
hence $e$ belongs to the limit $\lim\limits_{\nu\to 0^+}G_\mathbb{R}\cdot(\nu x)$,
whence $G_\mathbb{R}\cdot e\subset \lim\limits_{\nu\to 0^+}G_\mathbb{R}\cdot(\nu x)$ (see
Proposition \ref{basicprop}\,{\rm (a)}).

Up to $G_\mathbb{R}$-conjugation, we may assume that $\{h,e,f\}$ is a Cayley triple. 
Then, the element $x_e:=e-f$ belongs to $\mathfrak{k}_\mathbb{R}$, so it is semisimple and elliptic, and we have
\[
\mathrm{Ad}(\exp(th))(\nu x_e)=\nu(\exp(2t)e-\exp(-2t)f)\quad\mbox{for all $t\in\mathbb{R}$},
\]
hence if we take $t=\frac{-\ln\nu }{2}$
we conclude that $e\in \lim\limits_{\nu\to 0^+}G_\mathbb{R}\cdot(\nu x_e)$,
and therefore $G_\mathbb{R}\cdot e\subset \lim\limits_{\nu\to 0^+}G_\mathbb{R}\cdot(\nu x_e)$.
\end{proof}

\section{Asymptotic cones}\label{cones}

In this section, we define and relate various (complex or real) asymptotic cones that can be attached
to a subset $X$ of a (complex or real) vector space $V$ (Section \ref{section-5.1}). 
In the case where $X$ is the adjoint orbit of a semisimple element in the complex Lie algebra $\mathfrak{g}$,
we recall a result due to Borho and Kraft, which describes the asymptotic cone as the closure
of the Richardson orbit (Theorem \ref{TBK}).
In the case where $X$ is an adjoint orbit in the real Lie algebra $\mathfrak{g}_\mathbb{R}$,
the limit of orbits of Definition \ref{deflim} can be described as an asymptotic cone, and useful information can be deduced
from the result of Borho and Kraft (Section \ref{section-5.3}).

\subsection{General definitions of asymptotic cones}\label{section-5.1}

In this section, $V$ denotes a finite-dimensional complex vector space and $V_\mathbb{R}$ is a real vector space
realized as a real form of $V$.

We consider the extension $V\oplus\mathbb{C}$ of $V$. Both projective spaces $\mathbb{P}(V)$ and 
$\mathbb{P}(V\oplus\mathbb{C})$ are equipped with the Zariski topology. Consider the maps 
\begin{eqnarray*}
&&\iota:V\to\mathbb{P}(V\oplus\mathbb{C}),\ v\mapsto [v\oplus 1],\\
&&\kappa:\mathbb{P}(V)\to \mathbb{P}(V\oplus\mathbb{C}),\ [v]\mapsto[v\oplus 0],\\
&&\pi:V\setminus\{0\}\to\mathbb{P}(V),\; v\mapsto[v].
\end{eqnarray*}

\begin{Def} \label{complexAC}
For every subset $X\subset V$,
the {\em projective asymptotic cone for $X$} is defined as
$$\AC^\mathrm{p}(X)=\overline{\iota(X)}\cap\kappa(\mathbb{P}(V))$$
and the (affine) {\em asymptotic cone} is then the following closed cone in $V$:
\begin{equation*}
\AC(X)=\{0\}\cup\pi^{-1}(\kappa^{-1}(\AC^\mathrm{p}(X))).
\end{equation*}
\end{Def}

We adapt this construction to the real setting in the following way.
We consider the sphere 
\begin{equation*}
\mathbb{S}(V_\mathbb{R}):=(V_\mathbb{R}\setminus\{0\})/\mathbb{R}_+^*,
\end{equation*}
the standard embedding 
$$\iota^+:V_\mathbb{R}\hookrightarrow \mathbb{S}(V_\mathbb{R}\oplus\mathbb{R}),\; v\mapsto [v\oplus 1],$$ 
and the surjection 
$$\pi^+:V_\mathbb{R}\setminus\{0\}\to\mathbb{S}(V_\mathbb{R}).$$

\begin{Def}\label{realAC}
For a subset $X\subset V_\mathbb{R}$, we define the real asymptotic cone to be
\[
\AC_\mathbb{R}(X):=\{0\}\cup(\pi^+)^{-1}(\mathbb{S}(V_\mathbb{R})\cap\overline{\iota^+(X)}).
\]
Here we consider the analytic topology on
$V_\mathbb{R}\oplus\mathbb{R}$ and
the quotient topology on
$\mathbb{S}(V_\mathbb{R}\oplus\mathbb{R})$.
\end{Def}

\begin{rem}
The construction of the cone $\AC_\mathbb{R}(X)$ can also be described in the following explicit way.
\begin{itemize}
\item Consider $V_\mathbb{R}$ and $V_\mathbb{R}\oplus \mathbb{R}$ endowed with Euclidean norms,
so that $\mathbb{S}(V_\mathbb{R})$ and $\mathbb{S}(V_\mathbb{R}\oplus\mathbb{R})$ coincide with the usual spheres.
\item Note that $V_\mathbb{R}$ and $V_\mathbb{R}\oplus 1:=\{v\oplus 1\mid v\in V_\mathbb{R}\}$ are two parallel hyperplanes in $V_\mathbb{R}\oplus\mathbb{R}$. Project the subset $X\oplus 1\subset V_\mathbb{R}\oplus 1$ onto the sphere $\mathbb{S}(V_\mathbb{R}\oplus\mathbb{R})$ and consider the closure of this projection.
\item Intersect this closure with the hyperplane $V_\mathbb{R}$, and take the (nonnegative) cone of $V_\mathbb{R}$ determined by this intersection.
\end{itemize}
Figure \ref{F2} illustrates the construction in the case where $X$ is a conic.
\begin{figure}[H]
\includegraphics{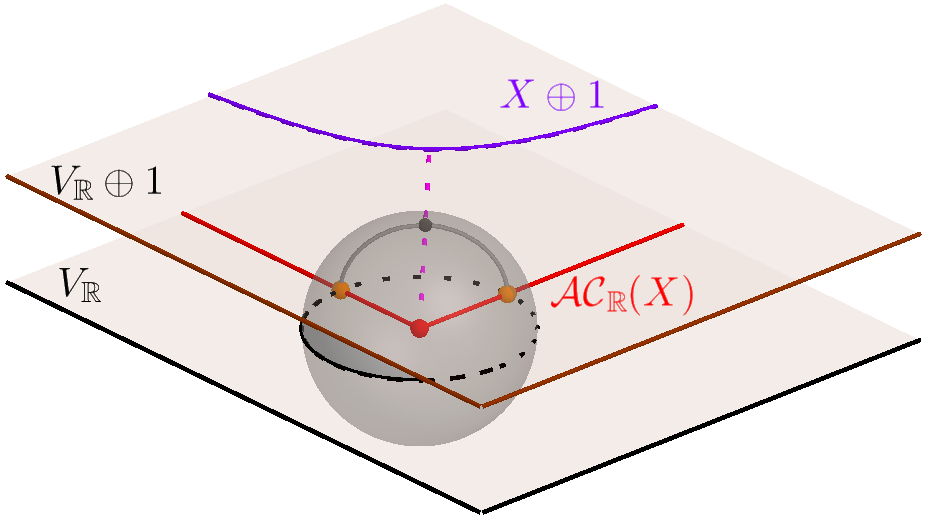}
\includegraphics{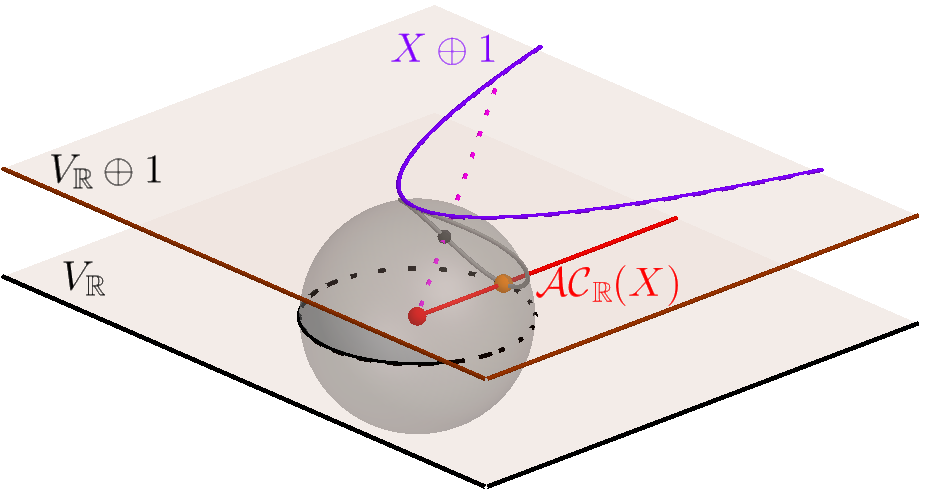}
\caption{The asymptotic cone of a hyperbola and of a parabola.}
\label{F2}
\end{figure}
\end{rem}

We point out a relation between the real and complex asymptotic cones, as well as
additional characterizations of the real cone. In particular, we shall see that $\AC_{\mathbb R}(X)$ coincides with the asymptotic cone $\ACC(X)$ introduced by Harris--He--\'Olafsson in \cite[Introduction]{HHO}:
\[
\ACC(X):=\{z\in V_{\mathbb R}\mid \Gamma\text{ an open cone containing $z$ }\Rightarrow \Gamma\cap X\text{ is unbounded}\}\cup\{0\}.
\]
\begin{pro}
\label{P-cones}
\begin{itemize}
\item[\rm (a)] For every subset $X\subset V_\mathbb{R}$, we have
\[
\AC_\mathbb{R}(X)\subset \AC(X).
\]
\item[\rm (b)] Let $X\subset V_\mathbb{R}$.
For every $z\in V_\mathbb{R}\setminus\{0\}$, the following conditions are equivalent:
\begin{itemize}
\item[\rm (i)] $z\in \AC_\mathbb{R}(X)$; 
\item[\rm (ii)] there are sequences $\{\nu_n\}_{n\geq 1}\subset\mathbb{R}_+^*$ converging to $0$
and $\{x_n\}_{n\geq 1}\subset X$ such that
$\displaystyle z=\lim_{n\to +\infty} \nu_n x_n$;
\item[\rm (iii)] for every open (positive) cone $\Gamma\subset V_\mathbb{R}$ which contains $z$, the intersection $\Gamma\cap X$ is unbounded.
\end{itemize}
In particular, $\AC_{\mathbb R}(X)$ and $\ACC(X)$ coincide:
\[
\ACC(X)=\AC_{\mathbb R}(X).
\]
\end{itemize}
\end{pro}

We also point out the following simple criterion for the asymptotic cone to be nontrivial:

\begin{pro} \label{cone-nontriv}
For every subset $X\subset V_\mathbb{R}$, the cone $\AC_\mathbb{R}(X)$ is nontrivial
(i.e., $\not=\{0\}$) if and only if the set $X$ is unbounded.
\end{pro}

\begin{proof}[Proof of Proposition \ref{P-cones}]
(a) The inclusion $V_\mathbb{R}\subset V$ yields a continuous mapping
$\phi:\mathbb{S}(V_\mathbb{R}\oplus\mathbb{R})\to\mathbb{P}(V\oplus\mathbb{C})$ with
$\phi(\mathbb{S}(V_\mathbb{R}))\subset\mathbb{P}(V)$. Moreover, we have a commutative diagram
\[
\xymatrix{
V_\mathbb{R}\setminus\{0\} \ar@{^{(}->}[d] \ar[r]^{\pi^+} & \mathbb{S}(V_\mathbb{R}) \ar@{^{(}->}[r] \ar[d]^\phi & 
\mathbb{S}(V_\mathbb{R}\oplus\mathbb{R}) \ar[d]^\phi & \ar[l]_{\iota^+} V_\mathbb{R} \ar@{^{(}->}[d] \\
V\setminus\{0\} \ar[r]^{\pi} & \mathbb{P}(V) \ar@{^{(}->}[r] & 
\mathbb{P}(V\oplus\mathbb{C}) & \ar[l]_{\iota} V
}
\]
For a subset $X\subset V_\mathbb{R}$, we get 
\[(\pi^+)^{-1}(\mathbb{S}(V_\mathbb{R})\cap\overline{\iota^+(X)})\subset (\pi^+)^{-1}(\phi^{-1}(\mathbb{P}(V))\cap
\phi^{-1}(\overline{\iota(X)}))=(V_\mathbb{R}\setminus\{0\})\cap\pi^{-1}(\mathbb{P}(V)\cap\overline{\iota(X)}). \] 
This yields the desired inclusion.

(b) Let $\pi^+$ also denote the canonical map 
\begin{equation*}
(V_\mathbb{R}\oplus\mathbb{R})\setminus\{0\}\to\mathbb{S}(V_\mathbb{R}\oplus\mathbb{R}).
\end{equation*}
By virtue of the continuity of $\pi^+$, for every neighborhood $\mathcal{W}$ of $\pi^+(z)$ in $\mathbb{S}(V_\mathbb{R}\oplus\mathbb{R})$, there are $\epsilon>0$ and a neighborhood $\mathcal{V}$ of $z$ in $V_\mathbb{R}$ such that $\pi^+(\mathcal{V}\oplus(-\epsilon,\epsilon))\subset\mathcal{W}$. Conversely, 
if $\epsilon>0$ and $\mathcal{V}$ is a neighborhood of $z$ in $V_\mathbb{R}$, then $\pi^+(\mathcal{V}\oplus(-\epsilon,\epsilon))$ is a neighborhood of $\pi^+(z)$ in $\mathbb{S}(V_\mathbb{R}\oplus\mathbb{R})$.

Condition (ii) is equivalent to
\begin{eqnarray}
\label{equivalent-ii}
\forall\epsilon>0,\ \forall \mathcal{V}\subset V_\mathbb{R}\mbox{ neighborhood of $z$,}\ \exists \nu\in(0,\epsilon),\ \exists x\in X,\mbox{ such that }\nu x\in\mathcal{V}.
\end{eqnarray}
By definition of the cone $\mathcal{AC}_{\mathbb{R}}(X)$, we have
\begin{eqnarray*}
z \in \AC_{\mathbb{R}}(X) & \Longleftrightarrow & \pi^+(z)\in\overline{\iota^+(X)} \\
 & \Longleftrightarrow & \forall\epsilon>0,\ \forall\mathcal{V}\subset V_\mathbb{R}\mbox{ neighborhood of $z$, }\pi^+(\mathcal{V}\oplus(-\epsilon,\epsilon))\cap \iota^+(X)\not=\emptyset.
\end{eqnarray*}
The intersection $\pi^+(\mathcal{V}\oplus(-\epsilon,\epsilon))\cap \iota^+(X)$ is nonempty if and only if there is $x\in X$ such that $\iota^+(x)=[x\oplus 1]$ belongs to $\pi^+(\mathcal{V}\oplus(-\epsilon,\epsilon))$, that is, if and only if there are $x\in X$ and $\nu\in\mathbb{R}_+^*$ such that \[\nu(x\oplus 1)\in\mathcal{V}\oplus(-\epsilon,\epsilon).\]
This is equivalent to saying that there are $x\in X$ and $\nu\in(0,\epsilon)$ such that $\nu x \in\mathcal{V}$. This establishes the equivalence between (i) and (ii).

(ii)$\Rightarrow$(iii): 
Let $\{\nu_n\}_{n\geq 1}\subset\mathbb{R}_+^*$ and $\{x_n\}_{n\geq 1}\subset X$ be as in (ii). Note that, since the sequence $\{\nu_n x_n\}_{n\geq 1}\subset V_\mathbb{R}$ has a nonzero limit while $\{\nu_n\}_{n\geq 1}$ converges to $0$, the sequence $\{x_n\}_{n\geq 1}$ is necessarily unbounded.
Now, let $\Gamma\subset V_\mathbb{R}$ be an open cone which contains $z$. 
Since $\Gamma$ is open, there is $n_0\geq 1$ such that, for all $n\geq n_0$, we have $\nu_nx_n\in \Gamma$.
Since $\Gamma$ is a cone, we get $x_n\in\Gamma\cap X$ for all $n\geq n_0$.
Therefore, the intersection $\Gamma\cap X$ is unbounded.
Therefore, (ii) implies (iii).

(iii)$\Rightarrow$(ii):
Let $\epsilon>0$ and let $\mathcal{V}\subset V_\mathbb{R}$ be an open, bounded neighborhood of $z$.
Let us consider the open cone $\Gamma:=\{tv\mid t\in\mathbb{R}_+^*,\ v\in\mathcal{V}\}$.
Since $z$ belongs to that cone, condition (iii) implies that there is an unbounded sequence
$\{x_n\}_{n\geq 1}$ contained in $\Gamma\cap X$.
The definition of $\Gamma$ yields a sequence $\{\nu_n\}_{n\geq 1}\subset\mathbb{R}_+^*$
such that $\nu_nx_n\in\mathcal{V}$ for all $n\geq 1$.
Since $\mathcal{V}$ is bounded while $\{x_n\}_{n\geq 1}$ is unbounded, the sequence $\{\nu_n\}_{n\geq 1}$ has to converge to $0$, hence we can find $n\geq 1$ such that $\nu_n\in(0,\epsilon)$.
We have shown the condition stated in (\ref{equivalent-ii}), and this implies that (ii) holds whenever (iii) is satisfied.
The proof of the proposition is complete.
\end{proof}

\begin{proof}[Proof of Proposition \ref{cone-nontriv}]
Let $\|\cdot\|$ be any norm on the finite-dimensional space $V_\mathbb{R}$.

Assume first that there exists $z\in\AC_\mathbb{R}(X)$, $z\not=0$.
By Proposition \ref{P-cones}\,{\rm (b)}, there is a sequence $\{\nu_n\}_{n\geq 1}\subset\mathbb{R}_+^*$
converging to $0$ and a sequence $\{x_n\}_{n\geq 1}\subset X$ such that 
\[\lim_{n\to +\infty}\nu_n x_n=z,\quad \mbox{hence}\quad
\lim_{n\to+\infty}\nu _n \|x_n\|=\|z\|\not=0.\]
This implies that $\|x_n\|\to +\infty$ as $n\to +\infty$. Therefore, $X$ contains an unbounded sequence.

Next assume that $X$ contains a sequence $\{x_n\}_{n\geq 1}$ such that $\|x_n\|\to+\infty$ as $n\to +\infty$.
Up to considering a subsequence, we may assume that the bounded sequence $\{\frac{x_n}{\|x_n\|}\}_{n\geq 1}$ converges to some $y\in V_\mathbb{R}$ such that $\|y\|=1$.
This yields
\[
\iota^+(x_n)=[x_n\oplus 1]=\big[\frac{x_n}{\|x_n\|}\oplus\frac{1}{\|x_n\|}\big]\longrightarrow [y\oplus 0]\quad\mbox{as $n\to+\infty$,}
\]
hence $\mathbb{S}(V_\mathbb{R})\cap\overline{\iota^+(X)}$ is nonempty as it contains $[y\oplus 0]$.
This implies that the cone $\AC_\mathbb{R}(X)$ is nontrivial as it contains the half line $\{ty\mid t\geq 0\}$.
\end{proof}

\subsection{Asymptotic cones of adjoint orbits}\label{section-5.3}

In this section, we take $V=\mathfrak{g}$ and $V_\mathbb{R}=\mathfrak{g}_\mathbb{R}$.
The asymptotic cone of the orbit of a semisimple element $x\in\mathfrak{g}$
can be described in the following way.

\begin{thm}[Borho--Kraft \cite{Borho-Kraft}] 
\label{TBK}
The asymptotic cone $\AC(G\cdot x)$ of the semisimple orbit $G\cdot x$ coincides with the Zariski closure of the Richardson orbit $\mathcal{O}_{\mathrm{Rich}}(x)$ defined in (\ref{DefRich}):
$$\AC(G\cdot x)=\overline{\mathcal{O}_{\mathrm{Rich}}(x)}.$$
\end{thm}

In the real setting, Proposition \ref{P-cones} yields the following result which relates the asymptotic cone with the limit of orbits.

\begin{thm}\label{posconethm}
\begin{itemize}
\item[(a)] For every $x\in\mathfrak{g}_\mathbb{R}$, we have
\[\lim_{\nu\to 0^+}G_\mathbb{R}\cdot(\nu x)=\AC_\mathbb{R}(G_\mathbb{R}\cdot x)
\subset \AC(G_\mathbb{R}\cdot x)
\subset \AC(G\cdot x).\]
\item[(b)] In particular, for every $x\in\mathfrak{g}_\mathbb{R}$ semisimple, we have
\[
\lim_{\nu\to 0^+}G_\mathbb{R}\cdot(\nu x)
\subset \overline{\mathcal{O}_{\mathrm{Rich}}(x)}\cap\mathfrak{g}_\mathbb{R}.\]
\end{itemize}
\end{thm}

\begin{proof}
The equality in (a) is implied by Proposition \ref{P-cones}\,(b)\,(i)$\Leftrightarrow$(ii).
The first inclusion is implied by Proposition \ref{P-cones}\,(a). The second inclusion is obvious.
(b) follows from (a) and Theorem \ref{TBK}.
\end{proof}

\section{Nontriviality of the limit of orbits}\label{nontriviality}

Relying on the identification of the limit of orbits with an asymptotic cone 
(Theorem \ref{posconethm})
and on the criterion stated in Proposition \ref{cone-nontriv} to guarantee the nontriviality of an asymptotic cone, we are in position to state the following result
regarding the nontriviality of the limit of orbits.

\begin{thm} \label{theorem-nontriv}
Recall that the group $G_\mathbb{R}$ is connected simple linear and non-compact.
Then, for every nonzero $x\in\mathfrak{g}_\mathbb{R}$, we have $\lim\limits_{\nu\to 0^+}G_\mathbb{R}\cdot(\nu x)\not=\{0\}$.
\end{thm}

\begin{proof}
Since $\mathrm{ad}(x)\not=0$, knowing that 
$\mathfrak{k}_\mathbb{R}=[\mathfrak{p}_\mathbb{R},\mathfrak{p}_\mathbb{R}]$
(because $\mathfrak{g}_\mathbb{R}$ is simple and non-compact; see \cite[VI.12.24]{Knapp}),
we can find $z\in\mathfrak{p}_\mathbb{R}$ such that $[x,z]\not=0$.
Let us write
\[
z=z_s+z_n,
\]
where $z_s$ is semisimple and hyperbolic, $z_n$ is nilpotent, $z_s,z_n\in\mathfrak{p}_\mathbb{R}$, and $[z_s,z_n]=0$.

If $[x,z_n]\not=0$, then
\[
\mathrm{Ad}(\exp(tz_n))(x)=x+t[z_n,x]+\sum_{k\geq 2}\frac{t^k}{k!}(\mathrm{ad}\,z)^k(x)
\]
determines a nonconstant polynomial curve contained in the orbit $G_\mathbb{R}\cdot x$,
which implies that $G_\mathbb{R}\cdot x$ is unbounded.

If $[x,z_n]=0$, we may assume that $z_n=0$. 
Then $z=z_s$ is semisimple and hyperbolic, and
the Lie algebra $\mathfrak{g}_\mathbb{R}$ decomposes as a sum of
eigenspaces
\[
\mathfrak{g}_\mathbb{R}=\bigoplus_{k=1}^\ell (\mathfrak{g}_\mathbb{R})_k
\quad\mbox{where}\quad
(\mathfrak{g}_\mathbb{R})_k:=\{y\in\mathfrak{g}_\mathbb{R}\mid [z,y]=\mu_k y\}
\]
for a sequence $\{\mu_1,\ldots,\mu_\ell\}$ of real numbers.
In particular we can write
\[
x=x_1+\ldots+x_\ell\quad\mbox{with }\ x_k\in(\mathfrak{g}_\mathbb{R})_k.
\]
Moreover, since $[z,x]\not=0$, there must be $k_0\in\{1,\ldots,\ell\}$
such that $x_{k_0}\not=0$ and $\mu_{k_0}\not=0$.
Then
\[
\mathrm{Ad}(\exp(tz))(x)=\sum_{k=1}^\ell \exp(t\mu_k)x_k\quad\mbox{(for $t\in\mathbb{R}$)}
\]
determines an unbounded curve contained in the orbit $G_\mathbb{R}\cdot x$.
In each case, we conclude that the orbit $G_\mathbb{R}\cdot x$ is unbounded.
The claim now follows from Proposition \ref{cone-nontriv} and Theorem \ref{posconethm}.
\end{proof}
\begin{rem}\label{remcpt}
(a) In the case where the Lie algebra $\mathfrak{g}_\mathbb{R}$ is compact, its nilpotent cone
is reduced to $\{0\}$, hence every limit of orbits is trivial.

(b) Theorem \ref{theorem-nontriv} extends to the case where $\mathfrak{g}_\mathbb{R}$ is semisimple,
with decomposition as a sum of simple ideals $\mathfrak{g}_\mathbb{R}=\mathfrak{g}_\mathbb{R}^1\oplus\ldots\oplus\mathfrak{g}_\mathbb{R}^\ell$,
not all of them being compact,
and $x=x^1+\ldots+x^\ell$ has at least one nonzero summand $x^i$ in a non-compact simple factor $\mathfrak{g}_\mathbb{R}^i$ of $\mathfrak{g}_\mathbb{R}$.
\end{rem}

Note that one can bound from above the dimension of the limit of orbits as\begin{equation}
\label{C1-new}
\dim \lim_{\nu\to 0^+}G_\mathbb{R}\cdot(\nu x)\leq \dim G_\mathbb{R}\cdot x.
\end{equation}
Moreover, the equality holds if and only if the Richardson orbit $\mathcal{O}_{\mathrm{Rich}}(x)$ intersects the limit $\lim\limits_{\nu\to 0^+}G_\mathbb{R}\cdot(\nu x)$.
This follows from (\ref{dim-Rich}), Theorem \ref{posconethm}, and the fact that $\dim_\mathbb{C}G\cdot x=\dim_\mathbb{R}G_\mathbb{R}\cdot x$.

It is however more difficult to bound from below the dimension of the limit.
Take for example $G_\mathbb{R}=\mathrm{SU}(n-1,1)$ with $n\geq 4$.
On one hand, if $x\in\mathfrak{su}(n-1,1)$ is regular, then $\dim G_\mathbb{R}\cdot x=n^2-n$.
On the other hand, every nilpotent element of $\mathfrak{su}(n-1,1)$ is of nilpotency order $\leq 3$
(see \cite[Theorem 9.3.3]{CM}) and its nilpotent orbit has dimension $\leq 4n-6$ (see \cite[Corollary 6.1.4]{CM}). 
Thus, in this case, the inequality in (\ref{C1-new}) is strict.


Even in cases in which every nilpotent orbit of the complexified Lie algebra $\mathfrak{g}$ admits real forms, equality does not necessarily hold in (\ref{C1-new}). 
Take for instance $G_\mathbb{R}=\mathrm{SU}(2,1)$ and 
$$x=\begin{pmatrix}-2i & 0 & 0 \\ 0 & -i & 0 \\ 0 & 0 & 3i\end{pmatrix}\in\mathfrak{su}(2,1).$$
The element $x$ is regular, semisimple (and elliptic), but
the limit $\lim\limits_{\nu\to 0^+}G_\mathbb{R}\cdot (\nu x)$ 
does not contain the sole principal nilpotent orbit of $\mathfrak{su}(2,1)$
as it only consists of elements of nilpotency order $\leq 2$; this follows from Section \ref{section-8.3} below.

\section{Limit of hyperbolic orbits}\label{hyperbolic}

In this section, we consider the limit of orbits associated with a nonzero hyperbolic semisimple element
$x\in\mathfrak{g}_\mathbb{R}$. 
This means that $\mathrm{ad}(x)$ has real eigenvalues on $\mathfrak{g}_\mathbb{R}$, 
thus also on $\mathfrak{g}$, and it determines a grading of the Lie algebra $\mathfrak{g}$
as in (\ref{grading-hyp}). Note that the eigenspace $\mathfrak{g}_\mu$ of the grading
is given by $\mathfrak{g}_\mu=(\mathfrak{g}_\mathbb{R})_\mu+i(\mathfrak{g}_\mathbb{R})_\mu$,
where $(\mathfrak{g}_\mathbb{R})_\mu$ stands for the eigenspace corresponding to the same eigenvalue
in $\mathfrak{g}_\mathbb{R}$.
Let
\[
\mathfrak{u}(x):=\bigoplus_{\mu>0}\mathfrak{g}_\mu\text{ and }\mathfrak{u}_\mathbb{R}(x):=\bigoplus_{\mu>0}(\mathfrak{g}_\mathbb{R})_\mu
\]
denote the nilradical of the parabolic subalgebra defined by $x$ and the corresponding real form.
By (\ref{Richardson-hyp}), the Richardson nilpotent orbit $\mathcal{O}_{\mathrm{Rich}}(x)$ corresponding to $x$
is such that
\[
\overline{\mathcal{O}_{\mathrm{Rich}}(x)}=G\cdot \mathfrak{u}(x).
\]
The following result provides a characterization of the limit of the orbits $G_\mathbb{R}\cdot(\nu x)$. The proof relies on a standard argument in Lie theory involving polarization of nilpotent orbits (see, e.g., \cite{Borho-Kraft}, \cite{Rossmann}).

\begin{thm}
\label{T4}
Let $x\in\mathfrak{g}_\mathbb{R}$ be a hyperbolic semisimple element.

{\rm (a)} The subset $G_\R\cdot\fu_\R(x)\subset\fg_\R$ is closed and equidimensional. It is of the form
\[
G_\R\cdot\fu_\R(x)=\overline{\mathcal{O}_1}\cup\ldots\cup\overline{\mathcal{O}_k}
\]
where $\mathcal{O}_1,\ldots,\mathcal{O}_k\subset\fg_\R$ are nilpotent $G_\R$-orbits which are all of the same dimension.

{\rm (b)} We have
\[
\lim_{\nu\to 0^+}G_\mathbb{R}\cdot(\nu x)=G_\mathbb{R}\cdot\mathfrak{u}_\mathbb{R}(x)=\overline{\mathcal{O}_1}\cup\ldots\cup\overline{\mathcal{O}_k}\subset \overline{\mathcal{O}_{\mathrm{Rich}}(x)}\cap\mathfrak{g}_\mathbb{R}.
\]
\end{thm}

\begin{proof}
(a) The fact that $G_\R\cdot\fu_\R(x)$ is closed follows from Lemma \ref{L6.2-new} below. Let
\[d=\max\{\dim G_\R\cdot z \mid z\in\fu_\R(x)\}.\]
We claim:
\begin{equation}
\label{6.4-new}
\mbox{the subset $Z_d:=\{z\in\fu_\R(x) \mid \dim G_\R\cdot z=d\}$ is open and dense in $\fu_\R(x)$.}
\end{equation}
Indeed, we first note that
\[
\dim G_\R\cdot z=\dim \fg_\R-\dim\{y\in\fg_\R \mid [z,y]=0\}=\rank(\ad(z))
\]
hence we have $\rank(\ad(z))\leq d$ for all $z\in\fu_\R(x)$ and
\[Z_d=\{z\in\fu_\R(x) \mid \rank(\ad(z))\geq d\}.\]
This implies that $Z_d$ is open in $\fu_\R(x)$. Moreover, if we fix $z_0\in Z_d$ (by assumption, $Z_d$ is nonempty) and take any $z\in \fu_\R(x)$, then we have $\rank(\ad(z+\epsilon z_0))\geq d$ whenever $\epsilon>0$ is small enough, hence $Z_d$ is dense in $\fu_\R(x)$. The claim made in (\ref{6.4-new}) is justified.

We infer that $G_\R\cdot Z_d$ is dense in $G_\R\cdot\fu_\R(x)$. By definition, $G_\R\cdot Z_d$ is $G_\R$-stable, nilpotent, hence it is a union of finitely many nilpotent orbits:
\[G_\R\cdot Z_d=\mathcal{O}_1\cup\ldots\cup\mathcal{O}_k,\qquad\mbox{whence}\quad G_\R\cdot\fu_\R(x)=\overline{G_\R\cdot Z_d}=\overline{\mathcal{O}_1}\cup\ldots\cup\overline{\mathcal{O}_k}.\]
Moreover, since $\mathcal{O}_j\subset Z_d$, we have $\dim\mathcal{O}_j=d$ for all $j\in\{1,\ldots,k\}$. The proof of part (a) is complete.

(b) The last inclusion is clear.
First we show that
\begin{equation}
\label{T4:inclusion-new}
\lim_{\nu\to 0^+}G_\mathbb{R}\cdot(\nu x)\subset G_\R\cdot \mathfrak{u}_\mathbb{R}(x).
\end{equation}
Let $P_\R\subset G_\R$ be the parabolic subgroup of Lie algebra $\fp_\R=\bigoplus\limits_{\mu\geq 0}(\fg_\R)_\mu$. Hence $\fu_\R(x)$ is the nilradical of $\fp_\R$.
First we note that
\[\lim_{\nu\to 0^+}G_\R\cdot(\nu x)=\bigcap_{\epsilon>0}\overline{\bigcup_{\nu\in(0,\epsilon)}G_\R\cdot(\nu x)}\subset \overline{\bigcup_{\nu\in(0,1)}G_\R\cdot(\nu x)}=\overline{G_\R\cdot \{\nu x \mid 0<\nu<1\}}.\]
Since the subset $\overline{P_\R\cdot\{\nu x \mid 0<\nu<1\}}\subset\fg_\R$ is closed and $P_\R$-stable, its image by $G_\R$ is closed (see Lemma \ref{L6.2-new}).
This yields
\[\lim_{\nu\to0^+}G_\R\cdot(\nu x)\subset G_\R\cdot\overline{P_\R\cdot\{\nu x \mid 0<\nu<1\}}.\]
Moreover, we have $P_\R=U_\R Z_{G_\R}(x)$, where $U_\R=\exp(\fu_\R(x))$ and $Z_{G_\R}(x)$ is the stabilizer of $x$ in $G_\R$. This implies that
\[P_\R\cdot\{\nu x \mid 0<\nu<1\}= U_\R\cdot\{\nu x \mid 0<\nu<1\}.\]
Every $u\in U_\R$ can be written as $u=\exp(z)$ with $z\in \fu_\R(x)$, and we have
\[\mathrm{Ad}(u)(x)=\sum_{k\geq 0}\frac{1}{k!}\mathrm{ad}(z)^k(x)=x+z'\quad\mbox{with}\quad z'\in\fu_\R(x)\]
since $\fu_\R(x)$ is both $\mathrm{ad}(x)$-stable and $\mathrm{ad}(\fu_\R(x))$-stable. Whence
\[P_\R\cdot\{\nu x \mid 0<\nu<1\}\subset \{\nu x+z' \mid 0<\nu<1,\ z'\in \fu_\R(x)\}.\]
Altogether, we get the inclusion
\[\lim_{\nu\to0^+}G_\R\cdot(\nu x)\subset G_\R\cdot\{\nu x+z' \mid 0\leq\nu\leq 1,\ z'\in\fu_\R(x)\}.\]
We have also that $\lim\limits_{\nu\to0^+}G_\R\cdot(\nu x)$ is contained in the nilpotent cone $\mathcal{N}(\fg_\R)$ (see Proposition \ref{basicprop}), and an element of the form $\nu x+z'$ can be nilpotent only if $\nu=0$. Therefore, (\ref{T4:inclusion-new}) is established.

It remains to show the inclusion
\begin{equation}
\label{T4:inclusion}
\mathfrak{u}_\mathbb{R}(x)\subset \lim_{\nu\to 0^+}G_\mathbb{R}\cdot(\nu x).
\end{equation}

We claim that
\begin{equation}
\label{T4:claim}
\{x+z\mid z\in\mathfrak{u}_\mathbb{R}(x)\}\subset G_\mathbb{R}\cdot x.
\end{equation}
Once we have shown (\ref{T4:claim}), we can deduce that, for all
$z\in\mathfrak{u}_\mathbb{R}(x)$, all $\nu>0$, the element $x+\frac{1}{\nu}z$ belongs to $G_\mathbb{R}\cdot x$,
so that $\nu x+z$ belongs to $G_\mathbb{R}\cdot(\nu x)$. This yields $z\in\lim\limits_{\nu\to 0^+}G_\mathbb{R}\cdot(\nu x)$. Since $z\in\mathfrak{u}_\mathbb{R}(x)$ is arbitrary, this establishes (\ref{T4:inclusion}). Therefore, it remains to show (\ref{T4:claim}).

For proving (\ref{T4:claim}), let $z\in\mathfrak{u}_\mathbb{R}(x)$. 
Let $(0<)\mu_1<\ldots<\mu_N$ be the list of all positive eigenvalues of $\mathrm{ad}(x)$.
We write
$z=z_1+z_2+\ldots+z_{N}$ 
with $z_j\in(\mathfrak{g}_\mathbb{R})_{\mu_j}$.
We construct an element $g\in G_\mathbb{R}$ such that $\mathrm{Ad}(g)(x)=x+z$.
Arguing by induction, we show that, for all $k=0,\ldots,N$, there is $g^{(k)}\in G_\mathbb{R}$
such that $\mathrm{Ad}(g^{(k)})(x)=x+z^{(k)}_1+\ldots+z^{(k)}_{N}$
with $z^{(k)}_{j}\in(\fg_\R)_{\mu_j}$ for all $j$ and $z^{(k)}_{j}=z_{j}$ whenever $j\leq k$.
If $k=0$, then $g^{(0)}=\mathrm{id}$ fulfills the required property.
Assume the construction done until rank $k<N$.
Let $y=z_{k+1}-z^{(k)}_{k+1}\in(\fg_\R)_{\mu_{k+1}}$.
Thus $[x,y]=\mu_{k+1}y$ and $[y,z_{j}^{(k)}]\in\bigoplus\limits_{\ell>k+1}(\fg_\R)_{\mu_\ell}$ for all $j\in\{1,\ldots,N\}$. 
Letting $t=\frac{-1}{\mu_{k+1}}$,
this implies that 
\begin{eqnarray*}
\mathrm{Ad}(\exp(ty)g^{(k)})(x) & = & \sum_{\ell\geq 0}\frac{t^\ell}{\ell!}\mathrm{ad}(y)^\ell( x+z^{(k)}_1+\ldots+z^{(k)}_{N}) \\
 & = & x+z^{(k)}_1+\ldots+z^{(k)}_{N}-t[x,y]+z' \\
 & = & x+z_1+\ldots+z_{k}+z_{k+1}+z''
\end{eqnarray*}
for some $z',z''\in\bigoplus\limits_{\ell>k+1}(\fg_\R)_{\mu_\ell}$.
This establishes the property at the rank $k+1$, and the proof of the theorem is complete.
\end{proof}

The above proof uses the following well-known fact. We give a proof for the sake of completeness.

\begin{lem}
\label{L6.2-new}
Let $P_\R\subset G_\R$ be a parabolic subgroup.
If $M\subset \fg_\R$ is a closed and $P_\R$-stable subset, then $G_\R\cdot M$ is closed.
\end{lem}

\begin{proof}
Let
$\{\mathrm{Ad}(g_k)(m_k)\}_{k\geq 1}$ be a sequence of elements of $G_\R\cdot M$, converging to some limit $m_0$. Since $G_\R/P_\R$ is compact, the sequence $\{g_kP_\R\}_{k\geq 1}$ has a convergent subsequence $\{g_{k_\ell}P_\R\}_{\ell\geq 1}$ with limit $g_0P_\R\in G_\R/P_\R$. Hence there is a sequence $\{p_\ell\}_{\ell\geq 1}\subset P_\R$ such that $g_0=\lim\limits_{\ell\to+\infty} g_{k_\ell}p_\ell$. 
Let $m'_\ell=\mathrm{Ad}(p_\ell^{-1})(m_{k_\ell})\in M$. Then
\[m'_\ell=\mathrm{Ad}(p_\ell^{-1})(m_{k_\ell})=\mathrm{Ad}(p_\ell^{-1}g_{k_\ell}^{-1})(\mathrm{Ad}(g_{k_\ell})(m_{k_\ell}))\stackrel{\ell\to+\infty}{\longrightarrow} \mathrm{Ad}(g_0^{-1})(m_0).\]
Since $M$ is closed, we deduce that $m_0\in\mathrm{Ad}(g_0)(M)$, hence $m_0\in G_\R\cdot M$.
\end{proof}



In the case of $\mathrm{SL}_n(\mathbb{R})$, we obtain the following refined description of the limit. Note that the inclusion in Theorem \ref{T4} becomes an equality in this case.

\begin{cor}
\label{corollary-SLn}
Assume that $G_\mathbb{R}=\mathrm{SL}_n(\mathbb{R})$.
Then, for all hyperbolic semisimple element $x\in\mathfrak{g}_\mathbb{R}$, we have
\[
\lim_{\nu\to 0^+}G_\mathbb{R}\cdot(\nu x)
=(G\cdot \mathfrak{u}(x))\cap\mathfrak{g}_\mathbb{R}=\overline{\mathcal{O}_{\mathrm{Rich}}(x)}\cap\mathfrak{g}_\mathbb{R}.
\]
\end{cor}

\begin{proof}
Up to conjugation, we may assume that $x\in\slnR$ is a diagonal matrix.

Recall that for every nilpotent $G$-orbit $\mathcal{O}$ of $\mathfrak{g}=\slnC$,
the intersection $\mathcal{O}\cap\slnR$ is nonempty and consists of exactly
one $\mathrm{GL}_n(\mathbb{R})$-orbit, which splits into at most two $\mathrm{SL}_n(\mathbb{R})$-orbits.
In particular, letting
\[
\sigma=\begin{pmatrix}
1 & 0 & \cdots & 0 \\ 0 & \ddots & \ddots & \vdots \\
\vdots & \ddots & 1 & 0 \\
0 & \cdots & 0 & -1
\end{pmatrix},
\]
we get that $\mathrm{Ad}(\sigma)$ induces an involution on the set of nilpotent $\mathrm{SL}_n(\mathbb{R})$-orbits of $\slnR$,
which switches the two real forms of $\mathcal{O}$ whenever $\mathcal{O}$ has two real forms,
and which stabilizes every orbit which is the sole real form of a complex nilpotent orbit.

It is well known that the Richardson orbit $\mathcal{O}_{\mathrm{Rich}}(x)$ has a representative in $\mathfrak{u}_\mathbb{R}(x)$;
see, for instance, the explicit construction made in \cite[\S7.2]{CM}.
Since $\mathfrak{u}_\mathbb{R}(x)$ is stable by $\mathrm{Ad}(\sigma)$, we deduce that each real form of 
$\mathcal{O}_{\mathrm{Rich}}(x)$ has a representative in $\mathfrak{u}_\mathbb{R}(x)$.
Moreover, if $\mathcal{O}$ is a nilpotent orbit of $\mathfrak{sl}_n(\mathbb{C})$ contained in the closure
of $\mathcal{O}_{\mathrm{Rich}}(x)$, then each real form of $\mathcal{O}$ is contained in the closure of
a real form of $\mathcal{O}_{\mathrm{Rich}}(x)$; this follows from \cite[Theorem 3]{Djokovic}.
The claimed equality ensues.
\end{proof}

\section{A mapping on the set of nilpotent orbits}\label{mapping}

Recall that, by Jacobson--Morozov theorem, every nonzero nilpotent element $e\in\mathfrak{g}_\mathbb{R}$ is 
part of an $\mathfrak{sl}_2$-triple $\{h,e,f\}\subset\mathfrak{g}_\mathbb{R}$. Moreover, two such triples
$\{h,e,f\}$ and $\{h',e,f'\}$ that contain $e$ are conjugate under $G_\mathbb{R}$.
This justifies the following definition.

\begin{Def}
\label{D5.1}
Let $e\in\mathfrak{g}_\mathbb{R}$ be a nonzero nilpotent element
and let $\mathcal{O}=G_\mathbb{R}\cdot e$ be its nilpotent orbit.
Given an $\mathfrak{sl}_2$-triple $\{h,e,f\}$ which contains $e$, we let
\[\mathcal{L}(e)=\lim_{\nu\to 0^+}G_\mathbb{R}\cdot(\nu h).\]
The set $\mathcal{L}(e)$ is a closed subset of the nilpotent cone $\mathcal{N}(\mathfrak{g}_\mathbb{R})$,
which is independent of the choice of $h$. Since $\mathcal{L}(e)$ is in fact independent of the choice of $e\in\mathcal{O}$, we set 
\begin{equation*}
\mathcal{L}(\mathcal{O}):=\mathcal{L}(e).
\end{equation*}
We also define $\mathcal{L}(\{0\})=\mathcal{L}(0)=\{0\}$.
\end{Def}

Note that $\mathcal{L}(\mathcal{O})$ always contains $\overline{\mathcal{O}}$
(see the proof of Proposition \ref{P-existence}).
Hence the definition of the mapping $\mathcal{L}$ is motivated by the fact that it provides us with a limit of semisimple orbits
which contains a specified nilpotent orbit $\mathcal{O}$. Note however that this mapping does not separate the orbit $\mathcal{O}=G_\mathbb{R}\cdot e$
from the other real forms of the complex orbit $G\cdot e$:

\begin{pro}
\label{mappingL-realforms}
For every nilpotent element $e\in\mathfrak{g}_\mathbb{R}$, we have
$(G\cdot e)\cap\mathfrak{g}_\mathbb{R}\subset\mathcal{L}(e)$.
\end{pro}

\begin{proof}
Let $e'\in(G\cdot e)\cap\mathfrak{g}_\mathbb{R}$, and
let $\{h,e,f\}$ and $\{h',e',f'\}$ be $\mathfrak{sl}_2$-triples in $\mathfrak{g}_\mathbb{R}$ containing $e$ and $e'$.
Since $G\cdot e=G\cdot e'$, these $\mathfrak{sl}_2$-triples are $G$-conjugate, hence so are in particular the
elements $h$ and $h'$.
These elements being semisimple and hyperbolic, they are in fact $G_\mathbb{R}$-conjugate; see, e.g., \cite[\S6]{Knapp} or \cite[Theorem 2.1]{Rothschild}.
This yields the equality $G_\mathbb{R}\cdot(\nu h)=G_\mathbb{R}\cdot(\nu h')$ for all $\nu>0$,
hence $\mathcal{L}(e)=\mathcal{L}(e')$, and therefore $e'\in\mathcal{L}(e)$.
\end{proof}

\begin{ex}
In Example \ref{E1} we have seen that the nilpotent cone $\mathcal{N}(\slR{2})$ consists
of three nilpotent orbits: $\{0\}$, $\mathfrak{O}^+=\mathrm{SL}_2(\mathbb{R})\cdot e$, and $\mathfrak{O}^-=\mathrm{SL}_2(\mathbb{R})\cdot f$, where 
\begin{equation*}
e=\begin{pmatrix} 0 & 1 \\ 0 & 0 \end{pmatrix}\;\;\text{ and }\;\;f=\begin{pmatrix} 0 & 0 \\ 1 & 0 \end{pmatrix}.
\end{equation*}
In this case, we have 
\begin{equation*}
\mathcal{L}(e)=\mathcal{L}(f)=\mathcal{N}(\mathfrak{sl}_2(\mathbb{R})).
\end{equation*}
\end{ex}

In Definition \ref{D5.1}, the semisimple element $h$ is in particular hyperbolic. Hence Theorem \ref{T4} gives information
on the nilpotent set $\mathcal{L}(e)$. As we show in the next subsection, this information is more precise in the case where the nilpotent element $e$ is even.

\subsection{The mapping $\mathcal{L}$ on even nilpotent orbits}

Recall from Section \ref{preliminaries} that an $\mathfrak{sl}_2$-triple $\{h,e,f\}$ is said to be even 
if the eigenvalues of $\mathrm{ad}(h)$ are all even integers; in this case, we also say that $e$ is an even nilpotent element 
and that $G_\mathbb{R}\cdot e$ is an even nilpotent orbit.

\begin{thm}
\label{T-P3}
Let $e\in\mathfrak{g}_\mathbb{R}$ be an even nilpotent element.
Then, we have
\[
\mathcal{L}(e)=\overline{(G\cdot e)\cap\mathfrak{g}_\mathbb{R}}.
\]
In other words, if $\mathcal{O}_1,\ldots,\mathcal{O}_r$ are the real forms of $G\cdot e$, then
\[
\mathcal{L}(e)=\overline{\mathcal{O}_1}\cup\ldots\cup\overline{\mathcal{O}_r}.
\]
\end{thm}

\begin{proof}
The inclusion $\supset$ follows from Proposition \ref{mappingL-realforms}.
For the second inclusion, let $\{h,e,f\}\subset\mathfrak{g}_\mathbb{R}$ be an even $\mathfrak{sl}_2$-triple
containing $e$.
We have
\[\fg=\bigoplus_{j\in\Z}\fg(2j)\supset\fg_\R=\bigoplus_{j\in\Z}\fg_\R(2j)\quad\mbox{where}\quad\fg(2j)=\{z\in\fg \mid [h,z]=2jz\}\supset\fg_\R(2j)=\{z\in\fg_\R \mid [h,z]=2jz\}.\]
As in Section \ref{hyperbolic}, we consider
\[\fu_\R(h)=\bigoplus_{j>0}\fg_\R(2j)\quad\mbox{and}\quad \fu(h)=\bigoplus_{j>0}\fg(2j).\]
The fact that $e$ is even implies that $G\cdot\fu(h)=\overline{G\cdot e}$.
The semisimple element $h$ is in particular hyperbolic,
and Theorem \ref{T4}, applied with $x=h$, yields the equalities
\[\mathcal{L}(e)=G_\R\cdot\fu_\R(h)=\overline{\mathcal{O}'_1}\cup\ldots\cup\overline{\mathcal{O}'_k},\]
where $\mathcal{O}'_1,\ldots,\mathcal{O}'_k\subset\fg_\R$ are nilpotent $G_\R$-orbits of the same dimension. We thus obtain:
\[\overline{(G\cdot e)\cap\fg_\R}\subset \mathcal{L}(e)=G_\R\cdot\fu_\R(h)\subset(G\cdot \fu(h))\cap\fg_\R=\overline{G\cdot e}\cap\fg_\R,\]
and we infer that $\mathcal{O}'_1,\ldots,\mathcal{O}'_k$ have to be the real forms of $G\cdot e$. The claimed equality $\mathcal{L}(e)=\overline{(G\cdot e)\cap\fg_\R}$ ensues.
\end{proof}

\begin{rem}
{\rm (a)}
In general $\mathcal{L}(e)=\overline{(G\cdot e)\cap\mathfrak{g}_\mathbb{R}}$ can be strictly contained in 
$\overline{G\cdot e}\cap\mathfrak{g}_\mathbb{R}$. 
Take for instance $G_\mathbb{R}=\mathrm{SU}(2,2)$. The nilpotent $\mathrm{SL}_4(\mathbb{C})$-orbit
$\mathcal{O}_\mathbf{d}$ parametrized by the partition $\mathbf{d}=(3,1)$ has two real forms, parametrized by the signed Young
diagrams (see \cite[Theorem 9.3.3]{CM} or Section \ref{section-8.3} below):
\[\young(+-+,-)\;\;\text{ and}\;\;\young(-+-,+).
\]
Moreover, $\mathcal{O}_\mathbf{d}$ is an even orbit. 
The nilpotent $\mathrm{SU}(2,2)$-orbits parametrized by the signed Young diagrams
\[
\young(+-,+-)\;\;\text{ and }\;\;\young(-+,-+)
\]
are both contained in $\overline{\mathrm{SL}_4(\mathbb{C})\cdot e}\cap\mathfrak{su}(2,2)$, but not in 
$\overline{(\mathrm{SL}_4(\mathbb{C})\cdot e)\cap\mathfrak{su}(2,2)}$; see, e.g., \cite{Djokovic}.

In the case of $G_\R=\mathrm{SL}_n(\R)$, however, the equality $\overline{(G\cdot e)\cap\fg_\R}=\overline{G\cdot e}\cap\fg_\R$ is true for all nilpotent element $e$, in view of the parametrization of real forms and the description of their closures given in \cite{Djokovic}; see also Theorem \ref{noneventhm2} below.

{\rm (b)}
Let us consider a symmetric subgroup $K\subset G$ compatible with the real form $G_\mathbb{R}$, and the corresponding decomposition $\mathfrak{g}=\mathfrak{k}\oplus\mathfrak{p}$. Nishiyama describes the (complex) asymptotic cone 
$\AC(K\cdot (i(e'-f')))$
where $\{h',e',f'\}$ is a normal triple with $e',f'$ even nilpotent \cite{Nishiyama}.
Thus Theorem \ref{T-P3} above appears to be a real counterpart of \cite[Theorem 0.2]{Nishiyama}. Specifically, in \cite{Nishiyama}, it is shown that the above asymptotic cone coincides with the closure of $(G\cdot e)\cap\mathfrak{p}$, and is in general
strictly contained in $\overline{G\cdot e}\cap\mathfrak{p}$. The equidimensionality of the complex cone (which is the counterpart of Theorem \ref{T4}\,{\rm (a)}) is used as a key ingredient in the proof of \cite[Theorem 0.2]{Nishiyama}. 
\end{rem}

\subsection{The mapping $\mathcal{L}$ on arbitrary nilpotent orbits}
Theorem \ref{T4} also gives information on $\mathcal{L}(e)$ in the case where $e$ is not even.
The key fact in the proof of Theorem \ref{T-P3} is that all the real forms of $G\cdot e$ have a representative in the space $\fu_\R(h)$ and, when $e$ is even, $G\cdot e=\mathcal{O}_{\mathrm{Rich}}(h)$ is in addition the largest nilpotent orbit that intersects $\fu_\R(h)$.

If $e\in\fg_\R$ is an arbitrary nilpotent element, then Theorem \ref{T4} and Proposition \ref{mappingL-realforms} imply that
\begin{equation}
\label{7.1-new}
\overline{(G\cdot e)\cap\fg_\R}\subset\mathcal{L}(e)=G_\R\cdot\fu_\R(h)\subset\overline{\mathcal{O}_{\mathrm{Rich}}(h)}\cap\fg_\R,
\end{equation}
where $\mathcal{L}(e)$ is equidimensional.
Here the difficulty is that, if $e$ is not even, then $G\cdot e$ does not coincide with the Richardson orbit $\mathcal{O}_{\mathrm{Rich}}(h)$, and we cannot guarantee that $\mathcal{O}_{\mathrm{Rich}}(h)$ has real forms or that these real forms have representatives in $\fu_\R(h)$.
We can however formulate the following statement.

\begin{pro}\label{PropReal}
Let $e\in\fg_\R$ be an arbitrary nilpotent element. Assume that the nilpotent $G$-orbit $\mathcal{O}_{\mathrm{Rich}}(h)$ has real forms which all intersect $\fu_\R(h)$. Then,
\[\mathcal{L}(e)=\overline{\mathcal{O}_{\mathrm{Rich}}(h)\cap\fg_\R}.\]
\end{pro}

\begin{proof}
The assumption made in the proposition combined with (\ref{7.1-new}) implies that
\[\overline{\mathcal{O}_{\mathrm{Rich}}(h)\cap\fg_\R}\subset\mathcal{L}(e)=G_\R\cdot\fu_\R(h)\subset \overline{\mathcal{O}_{\mathrm{Rich}}(h)}\cap\fg_\R.\]
Then, by arguing as in the proof of Theorem \ref{T-P3}, we can deduce the claimed equality by invoking the equidimensionality of $G_\R\cdot\fu_\R(h)$ (Theorem \ref{T4}).
\end{proof}

The condition in the proposition is fulfilled if $e$ is even (by Proposition \ref{mappingL-realforms}), in which case we simply retrieve the conclusion of Theorem \ref{T-P3}. The condition in the proposition is fulfilled for all $e$ in the case where $G_\R=\SLnR$,
as illustrated in the next subsection.

\subsection{Complete description of the mapping $\mathcal{L}$ in the case of $\mathrm{SL}_n(\mathbb{R})$}

In the case where $G_{\mathbb R}=\mathrm{SL}_n({\mathbb R})$, we can describe explicitly the mapping $\mathcal{L}$. Recall from Section \ref{preliminaries} that the nilpotent orbits $\mathcal{O}_\Lambda$ of $\mathfrak{g}=\mathfrak{sl}_n(\mathbb{C})$ are parametrized by partitions $\Lambda=(\Lambda_1,\ldots,\Lambda_m)$. Let $\Lambda^0$ (resp., $\Lambda^1$) be the subsequence of even (resp., odd) parts of $\Lambda$.
The nilpotent orbit $\mathcal{O}_\Lambda$ is even if and only if $\Lambda^0=\emptyset$ or $\Lambda^1=\emptyset$. We also recall that a (complex) nilpotent orbit $\mathcal{O}_\Lambda$ such that $\Lambda^1\not=\emptyset$ has only one real form $\mathcal{O}_\Lambda^\mathbb{R}$.

\begin{thm}
\label{noneventhm2}
Let $e,h,f\in\mathfrak{sl}_n(\mathbb{R})$ form a standard triple. 
Let $\mathcal{O}_\Lambda=\mathrm{SL}_n(\mathbb{C})\cdot e$.
\begin{itemize}
\item[\rm (a)] 
Assume that $e$ is even. Then $\mathcal{L}(e)=\overline{\mathcal{O}_\Lambda\cap\mathfrak{sl}_n(\R)}=\overline{\mathcal{O}_\Lambda}\cap\mathfrak{sl}_n(\mathbb{R})$.
\item[\rm (b)]
Assume that $e$ is not even, i.e., $\Lambda^0\not=\emptyset$ and $\Lambda^1\not=\emptyset$. Let $\Lambda^0+\Lambda^1:=(\mu^0_k+\mu^1_k)_{k=1}^\ell$ be the partition of $n$ obtained by summing the sequences $\Lambda^0=(\mu^0_k)_{k=1}^\ell$ and $\Lambda^1=(\mu^1_k)_{k=1}^\ell$ term by term (adding $0$'s if necessary, we may assume that both sequences have the same length $\ell$). Then one has:
\[
\mathcal{L}(\mathcal{O}_\Lambda^\mathbb{R})=\overline{\mathcal{O}_{\Lambda^0+\Lambda^1}^\mathbb{R}}.\]
\end{itemize}
\end{thm}

\begin{proof}
Part {\rm (a)} follows from Corollary \ref{corollary-SLn} and Theorem \ref{T-P3}. Let us show part {\rm (b)}.
We construct a semisimple element $h\in\slnR$ which belongs to an
$\mathfrak{sl}_2$-triple together with an element of
$\mathcal{O}_\Lambda^\mathbb{R}$. 
For an integer $d\geq 1$, let $D(d)$ denote the diagonal matrix of size $d$ whose coefficients
are $d-1,d-3,\ldots,-(d-1)$, and let $J(d)$ denote the Jordan matrix of size $d$ whose coefficients are $1$ just above the diagonal and $0$ elsewhere. We define $h,e$ as the blockwise diagonal matrices
\[
h=\begin{pmatrix}
D(\Lambda_1) \\
 & \ddots \\
 & & D(\Lambda_m)
\end{pmatrix}\;\;\;\;\text{ and }\;\;\;\;
e=\begin{pmatrix}
J(\Lambda_1) \\
 & \ddots \\
 & & J(\Lambda_m)
\end{pmatrix}.
\]
Then $e$ is an element in $\mathcal{O}_\Lambda^\mathbb{R}$ and we have $[h,e]=2e$.

We claim that $\mathcal{O}_{\mathrm{Rich}}(h)=\mathcal{O}_{\Lambda^0+\Lambda^1}$.
Once we have shown this equality, we deduce in particular that $\mathcal{O}_{\mathrm{Rich}}(h)$ has only one real
form (because the first part $\mu_1^0+\mu_1^1$ of the partition $\Lambda^0+\Lambda^1$ is odd),
and the result claimed in Theorem \ref{noneventhm2}\,{\rm (b)} finally follows from Corollary \ref{corollary-SLn}.

Recall that the Richardson nilpotent orbit $\mathcal{O}_{\mathrm{Rich}}(h)$ can be computed as follows.
For $d\in\mathbb{Z}$, let $N_d$ be the number of coefficients of $h$ which are equal to $d$.
We obtain a partition $\mu=(\mu_1,\mu_2,\ldots)\vdash n$ by letting
\[\mu_j:=|\{d\in\mathbb{Z}\mid N_d\geq j\}|\quad\mbox{for all $j\geq 1$}.\]
Then, by \cite[Theorem 7.2.3]{CM}, we have
$\mathcal{O}_{\mathrm{Rich}}(h)=\mathcal{O}_\mu$.
Hence, we have to check that $\mu=\Lambda^0+\Lambda^1$.

For $d$ even, we have $N_d=|\{j\,\mid\ |d|\leq \mu^1_j-1\}|$. 
For $d$ odd, we have $N_d=|\{j\,\mid\ |d|\leq \mu^0_j-1\}|$.
Hence, for $d$ even we have $N_d\geq j$ if and only if $|d|\leq \mu^1_j-1$,
whereas for $d$ odd we have $N_d\geq j$ if and only if $|d|\leq \mu^0_j-1$.
This yields
\[\mu_j=|\{d\mbox{ even}\,\mid\ |d|\leq\mu_j^1-1\}|+|\{d\mbox{ odd}\,\mid\ |d|\leq\mu_j^0-1\}|
=\mu_j^1+\mu^0_j\quad\mbox{for all $j$},\]
whence $\mu=\Lambda^0+\Lambda^1$.
\end{proof}

\begin{ex}
{\rm (a)} If $\Lambda=(8,7,6,4,4,3,2,1)$, then we get $\Lambda^0=(8,6,4,4,2)$,
$\Lambda^1=(7,3,1)$ and $\Lambda^0+\Lambda^1=(15,9,5,4,2)$.
Hence, in $\mathfrak{sl}_{35}(\mathbb{R})$, Theorem \ref{noneventhm2} yields
$\mathcal{L}(\mathcal{O}_{(8,7,6,4,4,3,2,1)}^\mathbb{R})=\overline{\mathcal{O}_{(15,9,5,4,2)}^\mathbb{R}}$.

{\rm (b)}
It is worth noting that the map $\mathcal{L}$ is not monotone nor injective. Indeed, in the Lie algebra $\mathfrak{g}_\mathbb{R}=\mathfrak{sl}_5(\mathbb{R})$, we have the following inclusions of orbit closures:
\[\overline{\mathcal{O}_{(2,1,1,1)}^\mathbb{R}}\subset \overline{\mathcal{O}_{(2,2,1)}^\mathbb{R}}\subset \overline{\mathcal{O}_{(3,1,1)}^\mathbb{R}}\subset
\overline{\mathcal{O}_{(5)}^\mathbb{R}}\]
while we have
\[
\mathcal{L}(\mathcal{O}_{(2,1,1,1)}^\mathbb{R})=
\mathcal{L}(\mathcal{O}_{(3,1,1)}^\mathbb{R})=\overline{\mathcal{O}_{(3,1,1)}^\mathbb{R}}\subset
\mathcal{L}(\mathcal{O}_{(2,2,1)}^\mathbb{R})=\overline{\mathcal{O}_{(3,2)}^\mathbb{R}}\subset
\mathcal{L}(\mathcal{O}_{(5)}^\mathbb{R})=\overline{\mathcal{O}_{(5)}^\mathbb{R}}.\]
\end{ex}

\section{Approximation of nilpotent orbits by elliptic orbits}\label{elliptic}

In this section, we consider the limit of orbits associated to a nonzero elliptic semisimple element $x\in\fg_\R$. 
Throughout this section, we assume that we are in the equal-rank situation: 
\begin{equation}
\label{ranks}
\rank(\fg_\R)=\rank(\fk_\R).
\end{equation}

Since $x$ is elliptic, $ix$ has real eigenvalues on $\fg$.
Then we can consider the decomposition 
$\fg=\bigoplus\limits_{\mu\in\R}\fg_\mu$ associated to $ix$ as in (\ref{grading-hyp}), and this gives rise to the parabolic subalgebra $\fq:=\bigoplus\limits_{\mu\geq 0}\fg_\mu$,
corresponding to a $\theta$-stable parabolic subgroup $Q\subset G$, and to the nilpotent radical
\begin{equation}
\label{uix}
\fu:=\fu(ix)=\bigoplus_{\mu>0}\fg_\mu.
\end{equation}
Recall the Cartan decomposition $\fg=\fk\oplus\fp$ of (\ref{cartandecomp}). Then 
\[
\mbox{$K\cdot(\fu\cap\fp)$ is a closed, irreducible, $K$-stable subvariety of $\fp$}
\]
(it is closed since $\fu\cap\fp$ is $K\cap Q$-stable, and $K\cap Q$ is a parabolic subgroup of $K$; it is irreducible since $K$ is connected). It follows that
\begin{equation}
\label{uniqueorbit}
\mbox{there is a unique nilpotent $K$-orbit $\mathcal{O}(ix)\subset\fp$ such that $K\cdot(\fu\cap\fp)=\overline{\mathcal{O}(ix)}$.}
\end{equation}
The limit of orbits of $x$ can now be characterized as the closure of the nilpotent $G_\R$-orbit of $\fg_\R$ corresponding to $\mathcal{O}(ix)$ through the Kostant--Sekiguchi bijection (see (\ref{KS})):
\begin{equation}
\label{formula-elliptic}
\lim_{\nu\to 0^+}G_\R\cdot(\nu x)=\overline{\KS^{-1}(\mathcal{O}(ix))}.
\end{equation}
This result can be obtained by combining several results from Geometric Representation Theory, that relate the limit of orbits to the asymptotic support and the associated variety of $(\fg,K)$-modules defined by parabolic induction: \cite[Proposition 3.7]{BV}, \cite[Proposition 3.4]{BV2}, and \cite[Proposition 5.4]{Trapa}. To our knowledge, there is no direct proof of the result stated in (\ref{formula-elliptic}), that relies on purely topological arguments.

In this section, we use the formula stated in (\ref{formula-elliptic}) for studying approximations of certain nilpotent orbits by elliptic semisimple orbits. In Section \ref{section-8.1}, we consider even nilpotent orbits.
In Section \ref{section-8.2}, we study approximation of minimal nilpotent orbits. In Section \ref{section-8.3}, we focus on the case where $G_\R=\mathrm{SU}(p,q)$.

\subsection{Approximation of even nilpotent orbits}

\label{section-8.1}

We give an ``elliptic analogue'' of the mapping $e\mapsto \mathcal{L}(e)$ constructed in Section \ref{mapping}.

\begin{Def}
Let $e\in\mathfrak{g}_\mathbb{R}$ be a nonzero nilpotent element
and let $\mathcal{O}=G_\mathbb{R}\cdot e$ be its nilpotent orbit.
Given an $\mathfrak{sl}_2$-triple $\{h,e,f\}$ which contains $e$, we let
\[\mathcal{L}'(e)=\lim_{\nu\to 0^+}G_\mathbb{R}\cdot(\nu (e-f)).\]
The set $\mathcal{L'}(e)$ is a closed subset of the nilpotent cone $\mathcal{N}(\mathfrak{g}_\mathbb{R})$. It is independent of the choice of the triple $\{h,e,f\}$, and in fact it is independent of $e$ up to its nilpotent orbit. We can therefore set 
\begin{equation*}
\mathcal{L'}(\mathcal{O}):=\mathcal{L'}(e).
\end{equation*}
We also define $\mathcal{L'}(\{0\})=\mathcal{L'}(0)=\{0\}$.
\end{Def}

As for $\mathcal{L}(\mathcal{O})$ in Section \ref{mapping}, the set $\mathcal{L}'(\mathcal{O})$ always contains $\overline{\mathcal{O}}$
(see the proof of Proposition \ref{P-existence}).
Hence the mapping $\mathcal{L}'$ also provides us with a limit of orbits
which contains the specified orbit $\mathcal{O}$. Unlike $\mathcal{L}(\mathcal{O})$ in general, the set $\mathcal{L}'(\mathcal{O})$ is irreducible: this follows from (\ref{formula-elliptic}), taking into account that the element $e-f$ is elliptic. In the case of even nilpotent elements, the following result (which is the ``elliptic analogue'' of Theorem \ref{T-P3}) implies that any even nilpotent orbit closure can be obtained as a limit of elliptic orbits.

\begin{thm}
\label{T-elliptic-even}
Assume that condition (\ref{ranks}) holds and that the nilpotent element $e\in\fg_\R$ is even. Then,
\[\mathcal{L}'(e)=\overline{G_\R\cdot e}.\]
\end{thm}

\begin{proof}
Let $\{h,e,f\}\subset\fg_\R$ be an $\mathfrak{sl}_2$-triple containing $e$. Up to dealing with $G_\R$-conjugates, we may assume that $\{h,e,f\}$ is a Cayley triple, and we denote by $\{h',e',f'\}\subset\fg$ its Cayley transform.
In particular, we have $h'=i(e-f)$, and $K\cdot e'$ is the image of $G_\R\cdot e$ by the Kostant--Sekiguchi correspondence (see (\ref{KS})).

Let $\fg=\bigoplus\limits_{j\in\Z}\fg_j$ be the decomposition of $\fg$ into eigenspaces for $\mathrm{ad}(h')$ and let $\fu=\bigoplus\limits_{j>0}\fg_j$. Since the triple $\{h,e,f\}$ is even, then so is $\{h',e',f'\}$, and we have
\[
\mathcal{O}_{\mathrm{Rich}}(h')=G\cdot e'=G\cdot e\quad\mbox{and}\quad
G\cdot \fu=\overline{\mathcal{O}_{\mathrm{Rich}}(h')}=\overline{G\cdot e'}=\overline{G\cdot e}.
\]
Let $\mathcal{O}\subset\fp$ be the unique nilpotent $K$-orbit which is dense in $K\cdot(\fu\cap \fp)$ (see (\ref{uniqueorbit})). By (\ref{formula-elliptic}), we have
\[
\mathcal{L}'(e)=\overline{\KS^{-1}(\mathcal{O})}.
\]

We claim that $\mathcal{O}=K\cdot e'$. Once we show this, we deduce that $\mathcal{L}'(e)=\overline{\KS^{-1}(K\cdot e')}=\overline{G_\R\cdot e}$, which will complete the proof of the theorem. For this,
we have 
\[\overline{\mathcal{O}}=K\cdot(\fu\cap\fp)\subset G\cdot\fu=\overline{G\cdot e'}.\]
On the other hand, since $e\in\mathcal{L}'(e)$, and since the Kostant--Sekiguchi correspondence preserves the closure relations, we must have $e'\in\overline{\mathcal{O}}$, hence 
\[
\overline{K\cdot e'}\subset\overline{\mathcal{O}}\subset \overline{G\cdot e'}.
\]
This implies that $\mathcal{O}$ is contained in $G\cdot e'$, and in fact $\mathcal{O}=K\cdot e'$. The proof of the theorem is now complete.
\end{proof}

\subsection{Approximation of minimal nilpotent orbits}

\label{section-8.2}

\label{section-minimal}

We still assume equal-rank situation (\ref{ranks}), so that the result stated in (\ref{formula-elliptic}) applies.
The assumption also implies that we can take a Cartan subalgebra of $\fg$ of the form $\mathfrak{h}=\mathfrak{h}_\mathbb{R}\oplus i\mathfrak{h}_\mathbb{R}\subset\mathfrak{k}$, where $\mathfrak{h}_\mathbb{R}\subset\mathfrak{k}_\mathbb{R}$.
Then the root system $\Phi=\Phi(\mathfrak{g},\mathfrak{h})$ decomposes as
\[
\Phi=\Phi_\mathrm{c}\sqcup\Phi_\mathrm{nc}
\]
so that we have the root space decompositions
\[
\mathfrak{k}=\mathfrak{h}\oplus\bigoplus_{\alpha\in\Phi_\mathrm{c}}\mathfrak{g}_\alpha
\qquad\mbox{and}\qquad
\mathfrak{p}=\bigoplus_{\alpha\in\Phi_\mathrm{nc}}\mathfrak{g}_\alpha.
\]
Moreover, since each root has pure imaginary values on $\mathfrak{h}_\mathbb{R}$, we have
$\overline{\mathfrak{g}_\alpha}=\mathfrak{g}_{-\alpha}$ for all root $\alpha$.
We consider a set of positive roots $\Phi^+$ and the corresponding set of simple roots
$\Delta$, which decomposes as $\Delta=\Delta_\mathrm{c}\sqcup\Delta_\mathrm{nc}$ in the same way as $\Phi$.

We focus on the case where $\mathfrak{g}$ is classical.
Moreover, we assume that the minimal nilpotent orbit $\mathcal{O}_{\mathrm{min}}\subset\mathfrak{g}$
has real forms, or equivalently that the intersection $\mathcal{O}_{\mathrm{min}}\cap\mathfrak{p}$ is nonempty.
Every root vector $e\in\mathfrak{g}_\alpha$ associated to a long root $\alpha$
is an element of $\mathcal{O}_{\mathrm{min}}$, 
and $\mathcal{O}_\mathrm{min}\cap\mathfrak{p}$ is nonempty if and only if
$\Phi_{\mathrm{nc}}$ contains a long root.
In this case, $\mathcal{O}_{\mathrm{min}}\cap\mathfrak{p}$ is the union of one or two $K$-orbits
(equivalently, $\mathcal{O}_\mathrm{min}$ has one or two real forms)
depending on whether the symmetric pair $(G,K)$ is non-Hermitian or Hermitian; see \cite{Okuda}.
The following table lists the cases that we are left to consider.
In each case, we indicate the number of real forms in $\mathcal{O}_\mathrm{min}$.
See \cite[\S VII.9]{Knapp} and \cite{Okuda} for more details.

\begin{center}
\begin{tabular}{|c|c|c|c|c|c|c|}
\hline 
type & $\mathfrak{g}$ & $\mathfrak{g}_\mathbb{R}$ & $\mathfrak{k}$ & $|\mathcal{O}_\mathrm{min}\cap\mathfrak{g}_\mathbb{R}/G_\mathbb{R}|$ 
\\ 
\hline
AIII & $\mathfrak{sl}_n(\mathbb{C})$, $n\geq 2$ & $\mathfrak{su}(p,q)$, $1\leq p\leq q$ & $\mathfrak{sl}_p(\mathbb{C})\times\mathfrak{sl}_q(\mathbb{C})\oplus\mathbb{C}$ &  2 
\\
BI & $\mathfrak{so}_{2n+1}(\mathbb{C})$, $n\geq 2$ & $\mathfrak{so}(2p,2q+1)$, $p,q\geq 1$ & 
$\mathfrak{so}_{2p}(\mathbb{C})\times\mathfrak{so}_{2q+1}(\mathbb{C})$ & $2$ if $p=1$, $1$ if $p\geq 2$ \\
CI & $\mathfrak{sp}_{2n}(\mathbb{C})$, $n\geq 2$ & $\mathfrak{sp}_{2n}(\mathbb{R})$ & $\mathfrak{gl}_n(\mathbb{C})$ &  $2$ 
\\
DI & $\mathfrak{so}_{2n}(\mathbb{C})$, $n\geq 4$ & $\mathfrak{so}(2p,2q)$, $1\leq p\leq q$ & $\mathfrak{so}_{2p}(\mathbb{C})\times\mathfrak{so}_{2q}(\mathbb{C})$ &  2 if $p=1$, 1 if $p\geq 2$ 
\\
DIII & $\mathfrak{so}_{2n}(\mathbb{C})$, $n\geq 4$ & $\mathfrak{so}_{2n}^*$ & $\mathfrak{gl}_n(\mathbb{C})$ &  2 
\\
\hline
\end{tabular}
\end{center}

\noindent
In each case, we point out that the set $\Delta_{\mathrm{nc}}$ consists of a single element $\alpha_0$,
which is always a long root. This is shown in the following table.

\begin{center}
\begin{tabular}{|c|c|c|c|c|c|c|c|c|}
\hline 
type & $\Phi^+$ & $\Phi^+\cap\Phi_{\mathrm{nc}}$ & $\alpha_0$ \\ 
\hline
AIII & $\{\varepsilon_i-\varepsilon_j\}_{1\leq i<j\leq n}$ & $\{\varepsilon_i-\varepsilon_j\}_{i\leq p<j}$ & $\varepsilon_p-\varepsilon_{p+1}$
\\
BI & $\{\varepsilon_i\pm\varepsilon_j\}_{1\leq i<j\leq n}\cup\{\varepsilon_i\}_{1\leq i\leq n}$ &
$\{\varepsilon_i\pm\varepsilon_j\}_{i\leq p<j}\cup\{\varepsilon_i\}_{i\leq p}$
 & $\varepsilon_p-\varepsilon_{p+1}$ 
\\
CI & $\{\varepsilon_i\pm\varepsilon_j\}_{1\leq i<j\leq n}\cup\{2\varepsilon_i\}_{1\leq i\leq n}$ & $\{\varepsilon_i+\varepsilon_j\}_{i\leq j}$ & $2\varepsilon_n$ 
\\
DI & $\{\varepsilon_i\pm\varepsilon_j\}_{1\leq i<j\leq n}$ & $\{\varepsilon_i\pm\varepsilon_j\}_{i\leq p<j}$ & $\varepsilon_p-\varepsilon_{p+1}$ 
\\
DIII & $\{\varepsilon_i\pm\varepsilon_j\}_{1\leq i<j\leq n}$ & $\{\varepsilon_i+\varepsilon_j\}_{i<j}$ & $\varepsilon_{n-1}+\varepsilon_n$ \\
\hline
\end{tabular}
\end{center}
We then focus on the real form
\[
\mathcal{O}_\mathrm{min}^+:=\mathrm{KS}^{-1}(K\cdot e_0)\subset\mathcal{O}_{\mathrm{min}}\quad\mbox{where}\quad e_0\in\mathfrak{g}_{\alpha_0}\setminus\{0\}.
\]

In each case, we consider the usual numbering of the simple roots, so that the first root
is $\alpha_1=\varepsilon_1-\varepsilon_2$.
Let $h_1:=\varpi_1^\vee\in\mathfrak{h}$ be the first fundamental coweight, characterized by
\[
\alpha_1(\varpi_1^\vee)=1\qquad\mbox{and}\qquad \alpha(\varpi_1^\vee)=0\ \mbox{ for all $\alpha\in\Delta\setminus\{\alpha_1\}$}.
\]
Since $\alpha(\varpi_1^\vee)$ is a real number for all root $\alpha$, we must have 
$\varpi_1^\vee\in i\mathfrak{h}_\mathbb{R}$, and therefore $x_1:=-i\varpi_1^\vee$ belongs to
$\mathfrak{h}_\mathbb{R}$ and is elliptic.

The following statement describes the limit of orbits $\lim\limits_{\nu\to 0^+}G_\mathbb{R}\cdot(\nu x_1)$
and shows in particular that it is close to the real nilpotent orbit $\mathcal{O}_{\mathrm{min}}^+$.
In the special case of type DI, we retrieve a result stated in \cite[Theorem 4.3]{KO}.
The statement uses the parametrization of complex nilpotent orbits of classical simple Lie algebras
by admissible partitions $\mathbf{d}=(d_1,\ldots,d_k)$, which is recalled in Section \ref{preliminaries}.
Note also that, in types A and C,
the minimal nilpotent orbit is parametrized by the partition $(2,1,\ldots,1)$,
while in types B and D, it is parametrized by $(2,2,1,\ldots,1)$.

\begin{thm}
\label{T-minimal}
With the above notation,
the limit of orbits $\lim\limits_{\nu\to 0^+}G_\mathbb{R}\cdot(\nu x_1)$ always contains the real nilpotent orbit 
$\mathcal{O}_\mathrm{min}^+$. Moreover:
\begin{itemize}
\item[\rm (a)] In types AIII and DIII, the limit $\lim\limits_{\nu\to 0^+}G_\mathbb{R}\cdot(\nu x_1)$ is
the closure of $\mathcal{O}_{\mathrm{min}}^+$.
\item[\rm (b)] In types BI, CI, and DI, we have
\[
\lim_{\nu\to 0^+}G_\mathbb{R}\cdot(\nu x_1)=\{0\}\cup\mathcal{O}_\mathrm{min}^+\cup\mathcal{O}_0,
\]
where $\mathcal{O}_0$ is a real form of the complex nilpotent orbit $\mathcal{O}_{\mathbf{d}}\subset\mathfrak{g}$
corresponding to the partition
\[
\mathbf{d}=\left\{
\begin{array}{ll}
(3,1,\ldots,1) & \mbox{in types BI and DI}, \\{}
(2,2,1,\ldots,1) & \mbox{in type CI}.
\end{array}
\right.
\]
\end{itemize}
\end{thm}

\begin{proof}
{}From (\ref{formula-elliptic}), we know that the limit $\lim\limits_{\nu\to 0^+}G_\mathbb{R}\cdot(\nu x_1)$
has a dense nilpotent $G_\mathbb{R}$-orbit $\mathcal{O}_0$; moreover, $\mathcal{O}_0$ is characterized by the fact that the corresponding $K$-orbit $\mathrm{KS}(\mathcal{O}_0)$ intersects the space $\mathfrak{u}\cap\mathfrak{p}$ along a dense open subset, where $\fu=\fu(\varpi_1^\vee)$ is given by (\ref{uix}). 

We have
\[
\mathfrak{u}\cap\mathfrak{p}=\bigoplus_{\beta\in\Psi_1}\mathfrak{g}_\beta
\]
where $\Psi_1:=\{\beta\in\Phi^+\cap \Phi_{\mathrm{nc}} \mid \beta(\varpi_1^\vee)>0\}$.
In the different cases, the set $\Psi_1$ can be described as follows.
\begin{center}
\begin{tabular}{|c|c|c|c|c|c|c|c|c|}
\hline
type & AIII & BI & CI & DI & DIII \\ \hline
$\Psi_1$ & $\{\varepsilon_1-\varepsilon_j\}_{j>p}$ & $\{\varepsilon_1\pm\varepsilon_j\}_{j>p}\cup\{\varepsilon_1\}$ & $\{\varepsilon_1+\varepsilon_j\}_{j\geq 1}$ & 
$\{\varepsilon_1\pm\varepsilon_j\}_{j>p}$ & $\{\varepsilon_1+\varepsilon_j\}_{j> 1}$ \\ \hline
\end{tabular}
\end{center}
Let
\[
\alpha'_1:=\left\{
\begin{array}{ll}
\varepsilon_1-\varepsilon_{p+1} & \mbox{in types AIII, BI, and DI,} \\
2\varepsilon_1 & \mbox{in type CI,} \\
\varepsilon_1+\varepsilon_n & \mbox{in type DIII.}
\end{array}
\right.
\]
In each case, $\alpha'_1$ is a long root which belongs to $\Psi_1$.
\begin{itemize}
\item In types AIII, BI, and DI, for $p=1$, we have $\alpha'_1=\alpha_0$.
\item In type CI, the product of transpositions $w=(1,n)(n+1,2n)$ is a Weyl group element which has a representant $g_w$ in $K$, and such that $w(\alpha_0)=\alpha'_1$. Whence $\mathrm{Ad}(g_w)(\fg_{\alpha_0})=\fg_{\alpha'_1}\subset\fu\cap\fp$.
\item In the other situations, $\alpha':=\alpha'_1-\alpha_0$ is a compact positive root such that
$\alpha'_1+k\alpha',\alpha_0-k\alpha'\notin\Phi\cup\{0\}$ for all $k\geq 1$.
Denoting by $\{u_{\pm\alpha'}(t)\}_{t\in\mathbb{C}}\subset K$ the one-parameter subgroup attached to $\pm\alpha'$, for every $e_0\in\mathfrak{g}_{\alpha_0}\setminus\{0\}$, 
we can find $t,s\in\mathbb{C}$ such that $\mathrm{Ad}(u_{-\alpha'}(t)u_{\alpha'}(s))(e_0)\in\mathfrak{g}_{\alpha'_1}\subset\mathfrak{u}\cap\mathfrak{p}$.
\end{itemize}
In all the cases, this implies that
\begin{equation}
\label{contains-Ominplus}
\mathcal{O}_\mathrm{min}^+\subset\lim_{\nu\to 0^+}G_\mathbb{R}\cdot(\nu x_1).
\end{equation}

In each type, the description of the set $\Psi_1$ allows us to determine the Jordan normal form of the elements
of the space $\mathfrak{u}\cap\mathfrak{p}$, viewed as matrices.
For each root $\beta$, let $E_\beta$ be a nonzero element of the root space $\mathfrak{g}_\beta$.
\begin{itemize}
\item 
In type AIII, any element of $\mathfrak{u}\cap\mathfrak{p}$ is a matrix of rank $\leq 1$.
The element $z_0:=E_{\varepsilon_1-\varepsilon_{p+1}}$ is of rank $1$.
\item 
In types CI and DIII, any element of $\mathfrak{u}\cap\mathfrak{p}$ is a matrix of rank $\leq 2$
and nilpotency order $\leq 2$.
The element $z_0:=E_{\varepsilon_1+\varepsilon_n}$ is of rank $2$ and nilpotency order $2$.
\item
In types BI and DI, any element of $\mathfrak{u}\cap\mathfrak{p}$ is a matrix of rank $\leq 2$
and nilpotency order $\leq 3$.
The element $z_0:=E_{\varepsilon_1-\varepsilon_{p+1}}+E_{\varepsilon_1+\varepsilon_{p+1}}$
is a matrix of rank $2$ and nilpotency order $3$.
\end{itemize}
In each case, we obtain that the $G$-orbit, and so the $K$-orbit, of the considered element $z_0$ 
intersects the space $\mathfrak{u}\cap\mathfrak{p}$ along a dense open subset.
In view of (\ref{formula-elliptic}), this yields
\[
\lim_{\nu\to 0^+}G_\mathbb{R}\cdot(\nu x_1)=\overline{\mathrm{KS}^{-1}(K\cdot z_0)}.
\]
In types AIII and DIII, the element $z_0$ actually belongs to $\mathcal{O}_\mathrm{min}$,
hence we must have $\mathcal{O}_\mathrm{min}^+=\mathrm{KS}^{-1}(K\cdot z_0)$.
In types BI, CI, and DI, the Jordan normal form of $z_0$ (viewed as a matrix) coincides
with the partition $\mathbf{d}$ indicated in the theorem. Moreover,
it follows from \cite{Djokovic,Ohta} that the closure of $\mathcal{O}_0:=\mathrm{KS}^{-1}(K\cdot z_0)$
is the union of $\mathcal{O}_0$, $\{0\}$, and only one real form of $\mathcal{O}_\mathrm{min}$,
which is precisely $\mathcal{O}_\mathrm{min}^+$ (see (\ref{contains-Ominplus})).
The proof  is complete.
\end{proof}

\subsection{Limit of elliptic orbits in the case of $G_\mathbb{R}=\mathrm{SU}(p,q)$} \label{section-8.3}

In this section, we assume that $G_\mathbb{R}=\mathrm{SU}(p,q)$
and $\mathfrak{g}_\mathbb{R}=\mathfrak{su}(p,q)$, with $p,q\geq 1$, $p+q=n$.
In this case, the complexification $\mathfrak{g}=\mathfrak{sl}_n(\mathbb{C})$
has  Cartan decomposition $\mathfrak{sl}_n(\mathbb{C})=\mathfrak{k}\oplus\mathfrak{p}$ with
$$
\mathfrak{k}=\left\{\begin{pmatrix} a & 0 \\ 0 & b \end{pmatrix} \mid a\in\mathfrak{gl}_p(\mathbb{C}),\ b\in\mathfrak{gl}_q(\mathbb{C}),\ \mathrm{Tr}(a)+\mathrm{Tr}(b)=0\right\}  \mbox{ \ and \ }  \mathfrak{p}=\left\{\begin{pmatrix} 0 & c \\ d & 0 \end{pmatrix} \mid c\in\mathrm{M}_{p,q}(\mathbb{C}),\ d\in\mathrm{M}_{q,p}(\mathbb{C})\right\}.
$$
Each elliptic orbit has a representative $x$ in $\mathfrak{k}_\mathbb{R}=\mathfrak{k}\cap\mathfrak{g}_\mathbb{R}$,
which is diagonalizable with pure imaginary eigenvalues.
The spectra of the diagonal blocks $ia,ib$ of $ix$ can be written as a pair of
nonincreasing sequences of real numbers
\begin{equation}
\label{Lambda-elliptic}
\Lambda=\big(
(\lambda_1\geq \lambda_2\geq \ldots\geq \lambda_p),(\mu_1\geq \mu_2\geq \ldots\geq \mu_q)
\big)\in\mathbb{R}^p\times\mathbb{R}^q\quad\mbox{such that }\sum_{k=1}^p\lambda_k+\sum_{\ell=1}^q\mu_\ell=0.
\end{equation}
Conversely, if $\Lambda$ is a pair of sequences as in (\ref{Lambda-elliptic}), 
then the diagonal matrix
\begin{equation}
\label{x-elliptic}
x:=\mathrm{diag}(-i\lambda_1,\ldots,-i\lambda_p,-i\mu_1,\ldots,-i\mu_q)
\end{equation}
is an elliptic element of $\mathfrak{g}_\mathbb{R}$ corresponding to $\Lambda$.

Let us recall from \cite[\S1.4]{Ohta} that
the nilpotent $K$-orbits of $\mathfrak{p}$ can be parametrized by the set of signed Young diagrams
with signature $(p,q)$:

A signed Young diagram $\sigma$ is obtained by filling in the boxes of a Young diagram
with signs $+$ and $-$ in such a way that the signs $+$ and $-$ alternate in each row (but not necessarily in each column). Moreover, two signed Young diagrams are considered equal if they coincide
up to permutation of the rows. The signature of $\sigma$ is the pair $(n_+,n_-)$, 
where $n_{\pm}$ indicates the number of $\pm$'s in $\sigma$. For example,
the set of signed Young diagrams of signature $(3,2)$ consists of the following list:

{\footnotesize
\[
\young(+-+-+)\,,\ \young(+-+-,+)\,,\ \young(-+-+,+)\,,\ \young(+-+,+-)\,,\ \young(+-+,-+)\,,\ \young(+-+,+,-)\,,\ \young(-+-,+,+)\,,\ \young(+-,+-,+)\,,\ \young(+-,-+,+)\,,\ \young(-+,-+,+)\,,\ \young(+-,+,+,-)\,,\ \young(-+,+,+,-)\,,\ \young(+,+,+,-,-)\,.
\]}

Let $V^+=\mathbb{C}^p\times\{0\}^q$ and $V^-=\{0\}^p\times \mathbb{C}^q$, so that 
$\mathfrak{p}$ can be viewed as the space of endomorphisms $z$ of $V=\mathbb{C}^n$
such that $z(V^+)\subset V^-$ and $z(V^-)\subset V^+$.
Given a signed Young diagram $\sigma$ with signature $(p,q)$,
let $\mathcal{O}_\sigma$ denote the set of nilpotent endomorphisms $z\in\mathfrak{p}$
which have a Jordan basis $(\varepsilon_c)$ parametrized by the boxes $c$ of $\sigma$, in such a way
that
\begin{itemize}
\item[\rm (C1)] the vector $\varepsilon_c$ belongs to $V^+$ if the box $c$ contains a sign $+$
and to $V^-$ if $c$ contains a $-$;
\item[\rm (C2)] if $c$ lies in the first column of $\sigma$, then $z(\varepsilon_c)=0$; otherwise,
$z(\varepsilon_c)=\varepsilon_{c'}$ where $c'$ is the box on the left of $c$.
\end{itemize}
The so-obtained subset $\mathcal{O}_\sigma$ is a $K$-orbit of $\mathfrak{p}$, and every
orbit is of this form for a unique signed Young diagram of signature $(p,q)$.

Next, we define a procedure to associate a signed Young diagram of signature $(p,q)$
to each $\Lambda$ as in (\ref{Lambda-elliptic}):

\begin{notation}
\label{N-sigmaLambda}
{\rm (a)}
If $\sigma$ is a signed Young diagram of signature $(p,q)$, then let $L_{r,s}*\sigma$ denote the signed Young diagram of signature $(p+r,q+s)$
obtained by adding a new box with sign $+$ on the left of each one of the first $r$ rows of $\sigma$ which start
with a $-$, and a new box with sign $-$ on the left of each one of the first $s$ rows of $\sigma$ which start with a $+$ (understanding that these new boxes can be inserted in empty rows if 
$\sigma$ has less than $r$ rows starting with a $-$ or less than $s$ rows starting with a $+$).
For example,
\[L_{2,1}*\mbox{\small $\young(+-+,-+,+,+)$}=\mbox{\small $\young(-+-+,+-+,+,+,+)$}.\]

{\rm (b)}
Whenever $\Lambda=\big((\lambda_1\geq\ldots\geq\lambda_p),(\mu_1\geq\ldots\geq\mu_q)\big)$ is a pair of sequences of real numbers, we define a signed Young diagram $\sigma_\Lambda$ of signature $(p,q)$ by induction:
\begin{itemize}
\item If $p=q=0$ ($\Lambda$ is an empty sequence), then let $\sigma_\Lambda$ be the empty signed Young diagram;
\item If $p>0$ or $q>0$, then set $\sigma_\Lambda=L_{r,s}*\sigma_{\Lambda'}$
where $(r,s)$ is such that
$\lambda_1=\ldots=\lambda_r=\mu_1=\ldots=\mu_s$ are the coefficients of $\Lambda$ equal to
$\max\Lambda$ (the maximal value in $\Lambda$) and $\Lambda':=((\lambda_{r+1},\ldots,\lambda_p),(\mu_{s+1},\ldots,\mu_q))$ is the pair of subsequences formed by the coefficients $<\max\Lambda$.
\end{itemize}
\end{notation}

\begin{ex}
For
\[
\Lambda=\big((6,3,2,2,-1,-3,-4),(3,2,-3,-3,-4)\big)
\]
we get
\begin{eqnarray*}
\sigma_\Lambda & = & L_{1,0}*L_{1,1}*L_{2,1}*L_{1,0}*L_{1,2}*L_{1,1}*\emptyset \ = \
 L_{1,0}*L_{1,1}*L_{2,1}*L_{1,0}*L_{1,2}*\young(+,-) \\
 & = & L_{1,0}*L_{1,1}*L_{2,1}*L_{1,0}*\mbox{\small $\young(-+,+-,-)$} \ = \
 L_{1,0}*L_{1,1}*L_{2,1}*\mbox{\small $\young(+-+,+-,-)$} \\
 & = & L_{1,0}*L_{1,1}*\mbox{\small $\young(-+-+,+-,+-,+)$} \ = \
 L_{1,0}*\mbox{\small $\young(+-+-+,-+-,+-,+)$} \ = \
\mbox{\small $ \young(+-+-+,+-+-,+-,+)$}.
\end{eqnarray*}
\end{ex}

\begin{thm}
\label{T-SUpq}
Assume that $G_\mathbb{R}=\mathrm{SU}(p,q)$ with $p,q\geq 1$.
\begin{itemize}
\item[\rm (a)] Let $x\in\mathfrak{g}_\mathbb{R}$ be the elliptic semisimple element
corresponding to a pair of sequences
$\Lambda=((\lambda_1\geq \ldots\geq \lambda_p),(\mu_1\geq\ldots\geq\mu_q))\in\mathbb{R}^p\times\mathbb{R}^q$ as in (\ref{Lambda-elliptic})--(\ref{x-elliptic}). Let $\fu=\fu(ix)$ be as in (\ref{uix}). Then, the unique nilpotent $K$-orbit which intersects
the space $\mathfrak{u}\cap\mathfrak{p}$ along a dense open subset is the orbit $\mathcal{O}_{\sigma_\Lambda}$ associated to the signed Young diagram $\sigma_\Lambda$.
Thus
\[
\lim_{\nu\to 0^+} G_\mathbb{R}\cdot(\nu x)=\overline{\mathrm{KS}^{-1}(\mathcal{O}_{\sigma_\Lambda})}
\]
(see (\ref{formula-elliptic})).
\item[\rm (b)] For every nilpotent $G_\mathbb{R}$-orbit $\mathcal{O}\subset\mathcal{N}(\mathfrak{g}_\mathbb{R})$, there is an elliptic orbit $G_\mathbb{R}\cdot x$ such that
$\lim\limits_{\nu\to 0^+} G_\mathbb{R}\cdot(\nu x)=\overline{\mathcal{O}}$.
\end{itemize}
\end{thm}

\begin{proof}
{\rm (a)} Let $(u_1,\ldots,u_p)$, resp. $(v_1,\ldots,v_q)$, be the standard basis
of $V^+$, resp. $V^-$.
Hence each vector $u_k$, resp. $v_\ell$, is an eigenvector of $x_0:=ix$ corresponding to the eigenvalue
$\lambda_k$, resp. $\mu_\ell$.
Viewing $\mathfrak{p}$ as above as the space of endomorphisms $z:V\to V$ such that $z(V^+)\subset V^-$ and $z(V^-)\subset V^+$,
the intersection $\mathfrak{u}\cap\mathfrak{p}$ coincides with the subspace of such endomorphisms
such that
\[
z(u_k)\in\langle v_j \mid \mu_j>\lambda_k\rangle,\qquad 
z(v_\ell)\in\langle u_j \mid \lambda_j>\mu_\ell\rangle\qquad\mbox{for all $k,\ell$}.
\]
In particular, for $r,s$ as in Notation \ref{N-sigmaLambda}\,{\rm (b)}, we have that
$u_1,\ldots,u_r,v_1,\ldots,v_s\in\ker z$.
Hence $z$ has the following matrix form:
\[
z=\left(
\begin{array}{ccc|ccc}
 & & & 0 & \multicolumn{2}{|c}{\gamma} \\
 \cline{5-6}
 & 0 & & \vdots & \multicolumn{2}{|c}{ } \\
 &  & & 0 & \multicolumn{2}{|c}{c'} \\
 \hline
0 & \multicolumn{2}{|c|}{\delta} & \\
\cline{2-3}
\vdots & \multicolumn{2}{|c|}{} & & 0&  \\
0 & \multicolumn{2}{|c|}{ d'} & & &
\end{array}
\right)\quad\mbox{with}\quad
\left\{
\begin{array}{l}
c'\in\mathrm{M}_{p-r,q-s}(\mathbb{C}),\ d'\in \mathrm{M}_{q-s,p-r}(\mathbb{C}), \\[2mm]
\gamma\in\mathrm{M}_{r,q-s}(\mathbb{C}),\ \delta\in\mathrm{M}_{s,p-r}(\mathbb{C}).
\end{array}
\right.
\]
Arguing by induction, we obtain that whenever $z$ is generic in $\mathfrak{u}\cap\mathfrak{p}$,
the submatrix
\[z':=\begin{pmatrix} 0 & c' \\ d' & 0 \end{pmatrix}\]
belongs to the nilpotent orbit parametrized by the signed Young diagram $\sigma_{\Lambda'}$
for $\Lambda'$ as in Notation \ref{N-sigmaLambda}\,{\rm (b)}.
This means that $z'$ has a Jordan basis parametrized 
by the boxes of $\sigma_{\Lambda'}$,
which satisfies the above conditions (C1)--(C2).
The fact that $z$ generically belongs to the orbit $\mathcal{O}_{\sigma_\Lambda}$
now follows straightforward from the definition of $\sigma_\Lambda$,
in view of the form of the matrix $z$ given above.

{\rm (b)} In view of part {\rm (a)}, the claim made in part {\rm (b)}
becomes equivalent to the claim that, for every
signed Young diagram $\sigma$ of signature $(p,q)$, we can find 
$\Lambda$ as in (\ref{Lambda-elliptic}) such that $\sigma=\sigma_\Lambda$.

Let $\ell$ be the number of columns in $\sigma$ and, for $k\in\{1,\ldots,\ell\}$, 
let $r_k$, resp. $s_k$, denote the number of signs $+$, resp. $-$, contained in the $k$-th column of
$\sigma$.
Thus $p=\sum\limits_{k=1}^\ell r_k$ and $q=\sum\limits_{k=1}^\ell s_k$.
Define the pair of sequences
\[
\Lambda=\big((N-1)^{r_1},\ldots,(N-\ell)^{r_\ell}),((N-1)^{s_1},\ldots,(N-\ell)^{s_\ell})\big),
\]
where the exponents $r_k$ and $s_k$ indicate the multiplicity of the coefficient $N-k$ in each sequence,
and $N$ is a number chosen so that the total sum of the coefficients in $\Lambda$
is equal to $0$, namely,
\[
N=\frac{1}{p+q}\sum_{k=1}^\ell k(r_k+s_k).
\]
It is easy to check (by induction) that, for every $k\in\{0,\ldots,\ell\}$,
the signed diagram
$L_{r_{\ell-k+1},s_{\ell-k+1}}*\cdots*L_{r_\ell,s_\ell}*\emptyset$ 
coincides with the subdiagram formed by the last $k$ columns of $\sigma$.
Whence
\[
\sigma_\Lambda=L_{r_1,s_1}*\cdots*L_{r_\ell,s_\ell}*\emptyset=\sigma.
\]
The proof of the theorem is complete.
\end{proof}

\begin{rem}
The calculation of $K\cdot(\fu\cap\fp)$ made in the above proof was done previously by Trapa \cite{Trapa}, though our combinatorial algorithm is somewhat different.
\end{rem}

\begin{ex}
{\rm (a)} The proof of Theorem \ref{T-SUpq} also indicates an explicit way to determine an elliptic element $x$ whose associated limit $\lim\limits_{\nu\to 0^+}G_\mathbb{R}\cdot(\nu x)$ coincides the closure of a specified
nilpotent orbit. 
For instance the signed Young diagram
\[
\sigma=\mbox{\small $\young(+-+-,+-,-+,+)$}
\]
has $\ell=4$ columns, and with the notation used in the proof we have
$(r_1,r_2,r_3,r_4)=(3,1,1,0)$ and $(s_1,s_2,s_3,s_4)=(1,2,0,1)$. In this case, we get
$\sigma=\sigma_\Lambda$ for the pair of sequences $\Lambda$ given by
\[
\Lambda=\big((8,8,8,-1,-10),(8,-1,-1,-19)\big).
\]
(In fact, every coefficient of the double sequence $\Lambda$ obtained through the algorithm presented in the proof of Theorem \ref{T-SUpq}\,{\rm (b)} should be multiplied by $\frac{1}{9}$, but we can avoid this factor.)

{\rm (b)} The elliptic element $-i\varpi_1^\vee\in\mathfrak{su}(p,q)$ induced by the first fundamental coweight $\varpi_1^\vee$
corresponds to the double sequence
\[
\Lambda=\big((\alpha,\beta^{p-1}),(\beta^q)\big)\quad\mbox{where}\quad\alpha=\frac{n-1}{n},\quad \beta=\frac{-1}{n}.
\]
Then we have
\[
\sigma_\Lambda=\mbox{\small $\young(+-,+,:,-,:)$},
\]
and $\mathrm{KS}^{-1}(\mathcal{O}_{\sigma_\Lambda})$ is in fact the real form $\mathcal{O}_\mathrm{min}^+$ considered
in Theorem \ref{T-minimal}. Thus we retrieve the result stated in Theorem \ref{T-minimal} (in type AIII).
\end{ex}


\end{document}